\documentclass[preprint,10pt]{elsarticle}
\usepackage{amsfonts,amssymb,amsmath,empheq}
\usepackage{amsthm}
\usepackage{booktabs}
\usepackage{graphics}
\usepackage{graphicx}
\usepackage{latexsym}
\usepackage{subfigure}
\usepackage{multirow}
\usepackage{natbib}
\usepackage{float}
\usepackage{cases}
\usepackage{bm}
\usepackage{color}
\usepackage{indentfirst}
\usepackage{epstopdf}
\usepackage{CJK}
\usepackage[usenames,dvipsnames]{xcolor}
\usepackage[colorlinks,
            linkcolor=red,
            anchorcolor=blue,
            citecolor=green
            ]{hyperref}
\newtheorem{thm}{Theorem}[section]

\newtheorem{lem}{Lemma}[section]

\newtheorem{rem}{Remark}[section]

\newcommand{\beqa}{\begin{eqnarray}}
\newcommand{\eeqa}{\end{eqnarray}}

\newcommand{\beq}{\begin{equation}}
\newcommand{\eeq}{\end{equation}}

\begin{document}

\begin{frontmatter}



\title{ Regularized finite difference methods for the logarithmic Klein-Gordon equation }


\author[nudt,nus]{Jingye Yan}
\ead{yanjingye0205@163.com}
\author[nudt]{Hong Zhang}
\ead{zhanghnudt@163.com}
\author[nudt]{Xu Qian\corref{cor1}}
\ead{qianxu@nudt.edu.cn}
\author[nudt,hpc]{Songhe Song}
\ead{shsong@nudt.edu.cn}

\cortext[cor1]{Corresponding author.}

\address[nudt]{Department of Mathematics, College of Arts and Sciences, National University of Defense Technology, Changsha 410073, China}
\address[nus]{Department of Mathematics, National University of Singapore, Singapore 119076}
\address[hpc]{State Key Laboratory of High Performance Computing, National University of Defense Technology, Changsha 410073, China}

\begin{abstract}
We propose and analyze two regularized finite difference methods for the logarithmic Klein-Gordon equation (LogKGE). Due to the blowup phenomena caused by the logarithmic nonlinearity of the LogKGE, it is difficult to construct numerical schemes and establish their error bounds. In order to avoid singularity, we present a regularized logarithmic Klein-Gordon equation (RLogKGE) with a small regularized parameter $0<\varepsilon\ll1$. Besides, two finite difference methods are adopted to solve the regularized logarithmic Klein-Gordon equation (RLogKGE) and rigorous error bounds are estimated in terms of the mesh size $h$, time step $\tau$, and the small regularized parameter $\varepsilon$. Finally, numerical experiments are carried out to verify our error estimates of the two numerical methods and the convergence results from the LogKGE to the RLogKGE with the linear convergence order $O(\varepsilon)$.
\end{abstract}


\end{frontmatter}


\section{Introduction}

The logarithmic  Klein-Gordon equation (LogKGE) known as the relativistic version of the logarithmic Schr\"{o}dinger equation \cite{bartkowski2008one} has been introduced in the quantum field theory by Rosen \cite{rosen1969dilatation} and has the form
\begin{align}\label{LogKGE}
\left\{
\begin{aligned}
&u_{tt}(\mathbf{x},t)-\Delta u(\mathbf{x},t)+u(\mathbf{x},t)+\lambda u(\mathbf{x},t)\ln(|u(\mathbf{x},t)|^{2})=0, ~~\mathbf{x}\in\mathbb{R}^{d}, t>0,\\
&u(\mathbf{x},0)=\phi(\mathbf{x}),~~~\partial_{t}u(\mathbf{x},0)=\gamma(\mathbf{x}),~~~\mathbf{x}\in\mathbb{R}^{d},
\end{aligned}
\right.
\end{align}
where $\mathbf{x}=(x_{1},\ldots,x_{d})^{\mathrm{T}}\in \mathbb{R}^{d}, (d =1, 2, 3)$ is the spatial coordinate, $t$ is time, $u:=u(\mathbf{x},t)$ is a real-valued scalar field, the parameter $\lambda$ measures the force of the nonlinear interaction.
 This kind of nonlinearity frequently appears in inflation cosmology and supersymmetric field theories \cite{rosen1969dilatation,enqvist1998q,linde1992strings}.
 The LogKGE (\ref{LogKGE}) has been used to describe the spinless particle \cite{sakurai1967advanced} in optics, nuclear physics and geophysics \cite{buljan2003incoherent,de2003logarithmic,hefter1985application,krolikowski2000unified}. Assume that $u(\cdot,t)\in H^{1}(\mathbb{R}^{d})$ and $\partial_{t}u(\cdot,t)\in L^{2}(\mathbb{R}^{d})$,
the LogKGE (\ref{LogKGE}) admits the energy conservation law \cite{masmoudi2002nonlinear,machihara2001nonrelativistic}, which is defined as:
\begin{align}
E(t) =\int_{\Omega}\left[ (u_{t}(\mathbf{x},t))^{2}+(\nabla u(\mathbf{x},t))^{2}+(1-\lambda)u^{2}(\mathbf{x},t)+\lambda u^{2}(\mathbf{x},t) \ln \left(|u(\mathbf{x},t)|^{2} \right) \right] \mathrm{d} x \equiv E(0). \label{energy}
\end{align}

The Klein-Gordon equation with logarithmic potentials posses some special analytical solutions in quantum mechanics \cite{maslov1990pulsons,bialynicki1979gaussons,koutvitsky2006instability}, and the existence of classical solutions and weak solutions have been investigated in \cite{bartkowski2008one,gorka2009logarithmic}. In the paper \cite{bialynicki1979gaussons}, the author studies the solutions named Gaussons which represent solutions of Gaussian shape  \cite{wazwaz2016gaussian}. Besides, the interaction of Gaussons has been introduced in \cite{makhankov1981interaction}. For the nonlinear Klein-Gordon equation (NKGE) and the oscillatory NKGE,  various analysis and numerical results have been represented in literature.
Along the mathematical front, the derivation, Cauchy problem, well-posedness and dynamical properties have been proposed in \cite{bainov1995nonexistence,brenner1981global,ibrahim2006global,kosecki1992unit,simon1993cauchy} and the references therein.
Along the numerical aspects, a surge of efficient and accurate numerical methods have been proposed and analyzed for the nonlinear Klein-Gordon equation (NKGE) and the oscillatory NKGE in the literature. For example, the standard finite difference time domain (FDTD) methods including energy conservative /semi-implicit /explicit finite difference time domain methods \cite{bao2012analysis,chang1991conservative,duncan1997sympletic,zhang2005convergence,bao2019long}, multiscale time integrator Fourier pseudospectral (MWI-FP) method \cite{bao2014uniformly}, finite element method \cite{cao1993fourier}, exponential wave integrator Fourier pseudospectral (EWI-FP) method \cite{bao2012analysis,bao2013exponential}, asymptotic preserving (AP) \cite{faou2014asymptotic} method, ect.  For numerical comparisons of different numerical methods of the NKGE and the oscillatory NKGE, we refer to \cite{bao2012analysis,jimenez1990analysis,pascual1995numerical,bao2019comparison}. However, due to the singularity of the logarithmic nonlinearity at the origin, these methods can not be applied to the LogKGE (\ref{LogKGE}) equation directly.

In order to avoid blowup of the LogKGE (\ref{LogKGE}), i.e., $\log |u|\rightarrow -\infty, |u|\rightarrow 0$, we consider a regularized logarithmic  Klein-Gordon equation (RLogKGE) with a small regularized parameter $0<\varepsilon \ll 1$,
\begin{align}\label{RLogKGE}
\left\{
\begin{aligned}
&u^{\varepsilon}_{tt}(\mathbf{x},t)-\Delta u^{\varepsilon}(\mathbf{x},t)+u^{\varepsilon}(\mathbf{x},t)+\lambda u^{\varepsilon}(\mathbf{x},t)\ln\left(\varepsilon^{2}+ \left( u^{\varepsilon}(\mathbf{x},t) \right)^{2} \right)=0, ~~\mathbf{x}\in\mathbb{R}^{d}, t>0,\\
&u^{\varepsilon}(\mathbf{x},0)=\phi(\mathbf{x}),~~~\partial_{t}u^{\varepsilon}(\mathbf{x},0)=\gamma(\mathbf{x}),~~~\mathbf{x}\in\mathbb{R}^{d}.
\end{aligned}
\right.
\end{align}
The above RLogKGE (\ref{RLogKGE}) is time symmetric or time reversible, i.e., they are invarient if interchanging $n+1\leftrightarrow n-1$ and $\tau \leftrightarrow -\tau$.
\begin{rem}
The Cauchy problem of the LogKGE (\ref{LogKGE}) and the RLogKGE (\ref{RLogKGE}), the convergence estimate between the regularized model (\ref{RLogKGE}) and the LogKGE (\ref{LogKGE}) will be represented in another paper.
\end{rem}
\begin{thm}
Assume $u^{\varepsilon}(\cdot,t)\in H^{1}(\mathbb{R}^{d})$ and $\partial_{t}u^{\varepsilon}(\cdot,t)\in L^{2}(\mathbb{R}^{d})$,
the RLogKGE (\ref{RLogKGE}) conserves energy conservation law
, which is defined as:
\begin{align}
E^{\varepsilon}(t) =\int_{\Omega}\left[ (u^{\varepsilon}_{t}(\mathbf{x},t))^{2}+(\nabla u^{\varepsilon}(\mathbf{x},t))^{2}+(u^{\varepsilon}(\mathbf{x},t))^{2}+\lambda F_{\varepsilon} \left( (u^{\varepsilon}(\mathbf{x},t))^{2} \right) \right] \mathrm{d} x \equiv E^{\varepsilon}(0), \label{energyr}
\end{align}
where $F_{\varepsilon}(\rho)=\int_{0}^{\rho} \ln (\varepsilon^{2}+s) \mathrm{d} s= \rho\ln (\varepsilon^{2}+\rho)+\varepsilon^{2}\ln (1+\frac{\rho}{\varepsilon^{2}})- \rho,   \quad \rho=(u^{\varepsilon}(\mathbf{x},t))^{2}.$
\end{thm}
\begin{proof}
\begin{align}
\begin{split}
\frac{d}{dt}E^{\varepsilon}(t)&=2\int_{\Omega}\left[ u^{\varepsilon}_{t}\cdot u^{\varepsilon}_{tt}+\nabla u^{\varepsilon}\cdot \nabla u^{\varepsilon}_{t}+u^{\varepsilon}u^{\varepsilon}_{t}+\lambda F^{'}_{\varepsilon}\left((u^{\varepsilon})^{2}\right)\cdot u^{\varepsilon}\cdot u^{\varepsilon}_{t}\right](\mathbf{x},t) \mathrm{d} x\\
&=2\int_{\Omega} \left[ u^{\varepsilon}_{t} \left( u^{\varepsilon}_{tt}-\Delta u^{\varepsilon}+u^{\varepsilon}+\lambda u^{\varepsilon}\ln\left(\varepsilon^{2}+( u^{\varepsilon})^{2} \right) \right) \right] (\mathbf{x},t) \mathrm{d} x=0.
\end{split}
\end{align}
This ends the proof.
\end{proof}

The main purpose of this work is to analyze two FDTD schemes for the RLogKGE (\ref{RLogKGE}) and study the efficiency, accuracy between the LogKGE (\ref{LogKGE}) and the RLogKGE (\ref{RLogKGE}) as well as their numerical simulations.

The rest of this paper is organized as follows. In Section 2, a semi-implicit and an explicit FDTD schemes are proposed for the RLogKGE (\ref{RLogKGE}). Besides, we analyze the stability and solvability of the two schemes. The details of error analysis are established in Section 3. Section 4 is devoted to verifying our error etimates using the numerical experiments. At last, some concluding remarks are drawn in Section 5. Throughout this paper, we denote $p\lesssim q$ to represent that there exists a generic
constant $C$ which is independent of $\tau,h,\varepsilon$, such that $|p|\leq Cq$.
\section{FDTD methods and their stability}
In this section, we construct two FDTD schemes to approximate the RLogKGE (\ref{RLogKGE}) and study their stability, solvability and analyze their error estimates. For simplicity of notations, we set $\lambda=1$ and only make analysis and construct numerical schemes in one dimensional space ($d=1$) for the RLogKGE (\ref{RLogKGE}). When $d=1$, we truncate the RLogKGE (\ref{RLogKGE}) with periodic boundary conditions
\begin{align}\label{RLogKGE1}
\left\{
\begin{aligned}
&u^{\varepsilon}_{tt}(x,t)-\Delta u^{\varepsilon}(x,t)+u^{\varepsilon}(x,t)+ u^{\varepsilon}(x,t)\ln\left(\varepsilon^{2}+ \left( u^{\varepsilon}(x,t) \right)^{2} \right)=0, ~~x\in\Omega=(a,b), ~~t>0,\\
&u^{\varepsilon}(x,0)=\phi(x),~~~\partial_{t}u^{\varepsilon}(x,0)=\gamma(x),~~~x\in\overline{\Omega}=[a,b].
\end{aligned}
\right.
\end{align}
\subsection{FDTD methods}
Choose time step $\tau:=\Delta t$ and time steps $t_{n}:=n\tau, n=0,1,2,\ldots;$  let the mesh size $h:=\frac{b-a}{N}$ with $N$ being a positive integer and denote the grid points as $x_{j}:=a+jh,j=0,1,\ldots,N$.
Define the index sets as:
\begin{align}
\mathcal{T}_{N}=\left\{j|j=0,1,2,\ldots,N-1\right\},\quad \mathcal{T}_{N}^{0}=\left\{j|j=0,1,2,\ldots,N\right\}.
\end{align}
Assume $u^{\varepsilon,n}_{j}, u^{n}_{j}$ are the approximations of the exact solution $u^{\varepsilon}(x_{j},t_{n})$ and $ u(x_{j},t_{n})$, $j\in \mathcal{T}_{N}^{0}$ and $n\geq 0$. Define $u^{\varepsilon,n}=(u^{\varepsilon,n}_{0},u^{\varepsilon,n}_{1},\ldots,u^{\varepsilon,n}_{N})^{\mathrm{T}}, u^{n}=(u^{n}_{0},u^{n}_{1},\ldots,u^{n}_{N})^{\mathrm{T}} \in \mathbb{C}^{N+1} $ as the numerical solutions vector at time $t=t_{n}$. The followings are the finite difference operators:
\begin{align*}
&\delta^{+}_{t}u^{n}_{j}=\frac{u^{n+1}_{j}-u^{n}_{j}}{\tau},~~~~~\delta_{t}^{-}u^{n}_{j}=\frac{u^{n}_{j}-u^{n-1}_{j}}{\tau},~~~~\delta^{2}_{t}u^{n}_{j}=\frac{u^{n+1}_{j}-2u^{n}_{j}+u^{n-1}_{j}}{\tau^{2}},\\
&\delta^{+}_{x}u^{n}_{j}=\frac{u^{n}_{j+1}-u^{n}_{j}}{h},~~~~~~\delta^{-}_{x}u^{n}_{j}=\frac{u^{n}_{j}-u^{n}_{j-1}}{h},~~~~~\delta^{2}_{x}u^{n}_{j}=\frac{u^{n}_{j+1}-2u^{n}_{j}+u^{n}_{j-1}}{h^{2}}.\\
\end{align*}
We denote a space of grid functions
\begin{align}
X_{N}=\left\{ u| u=(u_{0},u_{1},u_{2},\ldots,u_{N})^{\mathrm{T}},u_{0}=u_{N},u_{-1}=u_{N-1}\right\}\subseteq \mathbb{C}^{N+1}.
\end{align}
We define the standard discrete $l^{2}$, semi-$H^{1}$ and $l^{\infty}$ norms and inner product over $X_{N}$ as follows
\begin{align}
\|u\|_{l^{2}}^{2}=(u,u)=h\sum \limits^{N-1}_{j=0}|u_{j}|^{2},~~\|\delta_{x}^{+}u\|^{2}_{l^{2}}=h\sum \limits^{N-1}_{j=0}|\delta_{x}^{+}u_{j}|^{2},~~\|u\|_{l^{\infty}}=\sup\limits_{0\leq j\leq N-1}|u_{j}|,~~(u,v)=h\sum \limits^{N-1}_{j=0}u_{j}v_{j},
\end{align}
where $u,v\in X_{N}$, and $(\delta^{2}_{x}u,v)=-(\delta^{+}_{x}u,\delta^{+}_{x}v)=(u,\delta^{2}_{x}v)$.
In the following, we introduce two frequently used FDTD methods for the RLogKGE (\ref{RLogKGE}):\bigskip

$\mathbf{I}$. Semi-implicit finite difference (SIFD) scheme
\begin{align}\label{SIFD2}
\delta_{t}^{2}u^{\varepsilon,n}_{j}-\frac{1}{2}\delta_{x}^{2} (u^{\varepsilon,n+1}_{j}+u^{\varepsilon,n-1}_{j})+\frac{1}{2} (u^{\varepsilon,n+1}_{j}+u^{\varepsilon,n-1}_{j})+u^{\varepsilon,n}_{j}f_{\varepsilon}\left( (u^{\varepsilon,n}_{j})^{2} \right)=0,~~n\geq1;
\end{align}

$\mathbf{II}$. Explicit finite difference (EFD) scheme
\begin{align}\label{EFD}
\delta_{t}^{2}u^{\varepsilon,n}_{j}-\delta_{x}^{2} u^{\varepsilon,n}_{j}+u^{\varepsilon,n}_{j}+ u^{\varepsilon,n}_{j}f_{\varepsilon}\left( (u^{\varepsilon,n}_{j})^{2} \right)=0,~~n\geq1;
\end{align}
where, $f_{\varepsilon}(\rho)=\ln (\varepsilon^{2}+\rho)$.
 The initial and boundary conditions are discretized as
\begin{align}\label{initialvalue1}
u^{\varepsilon,n+1}_{0}=u^{\varepsilon,n+1}_{N},u^{\varepsilon,n+1}_{-1}=u^{\varepsilon,n+1}_{N-1},~~n\geq0,~~u^{\varepsilon,0}_{j}=\phi(x_{j}),~~j\in \mathcal{T}_{N}^{0}.
\end{align}
Using the Taylor expansion we can get the first step solution $u_{j}^{\varepsilon,1}$,
\begin{align}\label{initialvalue}
u_{j}^{\varepsilon,1}=\phi(x_{j})+\tau\gamma(x_{j})+\frac{\tau^{2}}{2}\left[\delta^{2}_{x}\phi(x_{j})-\phi(x_{j})- \phi(x_{j})\ln(\varepsilon^{2}+(\phi(x_{j}))^{2}) \right], ~~j\in \mathcal{T}_{N}^{0}.
\end{align}
It is easy to prove that the above FDTD schemes are all time symmetric or time reversible.
\subsection{Stability analysis}
Let $0<T<T_{\max}$ with $T_{\max}$ being the maximum existence time. Define
\begin{align}
\sigma_{\max}:=\max\{|\ln(\varepsilon^{2})|,\left| \ln(\varepsilon^{2}+\|u^{\varepsilon,n}\|^{2}_{l^{\infty}}) \right|\},~~0\leq n\leq\frac{T}{\tau}-1.
\end{align}
According to the von Neumann linear stability analysis, we can get the following stability results for the FDTD schemes.
\begin{thm}\label{stability}
For the above FDTD schemes applied to the RLogKGE (\ref{RLogKGE}) up to $t=T$, we have:

(i)~When $-1\leq\sigma_{\max}\leq 1$, the SIFD scheme (\ref{SIFD2}) is unconditionally stable; and when $\sigma_{\max}> 1$, it is conditionally stable under the stability condition
\begin{align}\label{sifd2stability}
\tau\leq \frac{2}{\sqrt{\sigma_{\max}-1}}.
\end{align}

(ii)~The EFD scheme (\ref{EFD}) is conditionally stable under the stability condition
\begin{align}\label{efdstability}
\tau\leq \frac{2h}{\sqrt{(\sigma_{\max}+1)h^{2}+4}}.
\end{align}
\end{thm}
\begin{proof}
Substituting
\begin{align}
u^{\varepsilon,n-1}_{j}=\sum_{l}\hat{U}_{l}e^{2ijl\pi/N},~~u^{\varepsilon,n}_{j}=\sum_{l}\xi_{l}\hat{U}_{l}e^{2ijl\pi/N},~~u^{\varepsilon,n+1}_{j}=\sum_{l}\xi_{l}^{2}\hat{U}_{l}e^{2ijl\pi/N},~~
\end{align}
into (\ref{SIFD2})-(\ref{EFD}), where $\xi_{l}$ is the amplification factor of the $l$th mode in phase space, we can get the characteristic equation with the following structure
\begin{align}
\xi_{l}^{2}-2 \theta_{l} \xi_{l}+1=0, \quad l=-\frac{N}{2}, \ldots, \frac{N}{2}-1,
\end{align}
where $\theta_{l}$ is invarient with different methods. By the above equation, we get $\xi_{l}=\theta_{l}\pm \sqrt{\theta_{l}^{2}-1}$. The stability of numerical schemes amounts to
\begin{align}
\left|\xi_{l}\right| \leq 1 \Longleftrightarrow\left|\theta_{l}\right| \leq 1, \quad l=-\frac{N}{2}, \ldots, \frac{N}{2}-1.
\end{align}
Denote $s_{l}=\frac{2}{h} \sin \left(\frac{l \pi}{N}\right), \quad l=-\frac{N}{2}, \ldots, \frac{N}{2}-1$, we have
\begin{align}\label{stability1}
0\leq s_{l}^{2}\leq \frac{4}{h^{2}}.
\end{align}
Firstly, we prove linear stability. We assume $f_{\varepsilon}\left((u^{\varepsilon})^{2}\right)=\alpha$, and $\alpha$ is a constant satisfying $\alpha>-1$, then (\ref{SIFD2}) and (\ref{EFD}) are linear.

(i) For the SIFD scheme (\ref{SIFD2}), we have
\begin{align}
\theta_{l}=\frac{2- \alpha\tau^{2}}{2 +\tau^{2}\left(s_{l}^{2}+1\right)} , \quad l=-\frac{N}{2}, \ldots, \frac{N}{2}-1.
\end{align}
When $-1\leq\alpha\leq 1$, it implies that $|\theta_{l}| \leq 1$ and the SIFD scheme (\ref{SIFD2}) is unconditional stable. On the other hand, when $1<\alpha$, we have
\begin{align}
2-\alpha\tau^{2}\geq-2-\tau^{2},
\end{align}
it implies that, when $\tau\leq \frac{2}{\sqrt{\alpha-1}}$, the SIFD scheme (\ref{SIFD2}) is stable.

And when SIFD scheme is nonlinear, with the same method we can get
stability condition is
\begin{align}
\tau\leq \frac{2}{\sqrt{\sigma_{\max}-1}}.
\end{align}

(ii) For the EFD scheme (\ref{EFD}), we have
\begin{align}
\theta_{l}=\frac{2- \tau^{2}(1+\alpha+s_{l}^{2})}{2} , \quad l=-\frac{N}{2}, \ldots, \frac{N}{2}-1.
\end{align}
By (\ref{stability1}), we get
\begin{align}
\tau^{2}(\alpha+1+s_{l}^{2})\leq \tau^{2}(\alpha+1+\frac{4}{h^{2}})\leq 4, \Rightarrow |\theta_{l}|<1.
\end{align}
It implies that, when $\tau\leq \frac{2h}{\sqrt{(\alpha+1)h^{2}+4}}$, the EFD scheme (\ref{EFD}) is stable. Besides, when the EFD scheme (\ref{EFD}) is nonlinear, the stability condition is
\begin{align}
\tau\leq \frac{2h}{\sqrt{(\sigma_{\max}+1)h^{2}+4}}.
\end{align}
\end{proof}

\begin{rem}
Since the scheme SIFD (\ref{SIFD2}) is linear, and the coefficient matrix is strictly diagonal, it is easy to conclude that the SIFD (\ref{SIFD2}) is solvable. In addition, (\ref{EFD}) is explicit, so it is evident that there exists a unique solution.
\end{rem}
\section{Error esitimates}
\subsection{Main results}
Motivated by the analytical results in \cite{bao2019long,bao2012analysis,bao2019error,bao2018uniform,bao2018uniforme}, we will establish the error estimates of the FDTD schemes.

Assume that the solution $u^{\varepsilon}$ is smooth enough over $\Omega_{T}:\Omega\times [0,T],$ i.e.
\begin{align}
(A)~~~~~u^{\varepsilon}\in C\left( [0,T]; H^{5}(\Omega)\right)\cap C^{2}\left( [0,T]; H^{4}(\Omega)\right)\cap C^{4}\left( [0,T]; H^{2}(\Omega)\right),
\end{align}
and there exsit $\varepsilon_{0} >0$ and $C_{0}>0$ independent of $\varepsilon$ such that
\begin{align}
\left\|u^{\varepsilon}\right\|_{L^{\infty}\left(0, T ; H^{5}(\Omega)\right)}+\left\|\partial_{t}^{2} u^{\varepsilon}\right\|_{L^{\infty}\left(0, T ; H^{4}(\Omega)\right)}+\left\|\partial_{t}^{4} u^{\varepsilon}\right\|_{L^{\infty}\left(0, T ; H^{2}(\Omega)\right)} \leq C_{0},
\end{align}
is valid uniformly in $0\leq \varepsilon \leq\varepsilon_{0}$.

Denote $\Lambda=\|u^{\varepsilon}(x,t)\|_{L^{\infty}(\Omega_{T})}$ and the grid `error' function $e^{\varepsilon,n} \in X_{N}~(n\geq0)$ as
\begin{align}
e^{\varepsilon,n}_{j}=u^{\varepsilon}(x_{j},t_{n})-u^{\varepsilon,n}_{j},~~j\in \mathcal{T}^{0}_{N},~~n=0,1,2,\ldots,
\end{align}
where $u^{\varepsilon}$ and $u^{\varepsilon,n}_{j}$ are the exact solution and numerical approximation of (\ref{RLogKGE1}) respectively.
\begin{thm}\label{SIFD2err}
Assume $\tau\lesssim h$ and under the assumption (A), there exist $h_{0}>0,\tau_{0}>0$  sufficiently small and independent of $\varepsilon,$ for any $0<\varepsilon \ll1$, when $0<h\leq h_{0}$ and $0<\tau\leq \tau_{0}$ and under the stability condition (\ref{sifd2stability}),  the SIFD (\ref{SIFD2}) with (\ref{initialvalue1}) and (\ref{initialvalue}) satisfies the following error estimates
\begin{align}\label{SIFD2error}
\|\delta_{x}^{+}e^{\varepsilon,n}\|_{l^{2}} + \|e^{\varepsilon,n}\|_{l^{2}}\lesssim e^{\frac{T}{2}(\ln(\varepsilon^{2}))^{2}}(\tau^{2}+h^{2}),~~\|u^{\varepsilon,n}\|_{l^{\infty}}\leq \Lambda+1.
\end{align}
\end{thm}
\begin{thm}\label{EFDerr}
Assume $\tau \leq \frac{1}{2}\min\{1,h\}$ and under the assumption (A), there exist $h_{0}>0,\tau_{0}>0$  sufficiently small and independent of $\varepsilon,$  for any $0<\varepsilon \ll1$, when $0<h\leq h_{0}$ and $0<\tau\leq \tau_{0}$  and under the stability condition (\ref{efdstability}), the EFD (\ref{EFD}) with (\ref{initialvalue1}) and (\ref{initialvalue}) satisfies the error estimates
\begin{align}\label{EFDerror}
\|\delta_{x}^{+}e^{\varepsilon,n}\|_{l^{2}} + \|e^{\varepsilon,n}\|_{l^{2}}\lesssim e^{\frac{T}{2}(\ln(\varepsilon^{2}))^{2}}(\tau^{2}+h^{2}),~~\|u^{\varepsilon,n}\|_{l^{\infty}}\leq \Lambda+1.
\end{align}
\end{thm}
\begin{rem} \cite{bao2012analysis,bao2019long}
Extending to $2$ and $3$ dimensions, the above Theorems are still valid under the conditions $0<h\lesssim \sqrt{C_{d}(h)},~0<\tau\lesssim \sqrt{C_{d}(h)}$. Besides, the inverse inequality becomes
\begin{align}
\|u^{\varepsilon,n}\|_{l^{\infty}}\lesssim \frac{1}{C_{d}(h)}\left(\|\delta_{x}^{+}u^{\varepsilon,n}\|_{l^{2}} + \|u^{\varepsilon,n}\|_{l^{2}} \right),
\end{align}
where $C_{d}(h)=1/|\ln h|$ when $d=2$ and when $d=3$, $C_{d}(h)=h^{1/2}$.
\end{rem}
\subsection{Proof of Theorem \ref{SIFD2err} for the SIFD}
Define the local trunction error for the SIFD (\ref{SIFD2}) as
\begin{align}\label{trunerre}
\begin{split}
\xi^{0}_{j}:=&\delta_{t}^{+}u^{\varepsilon}(x_{j},0)-\gamma(x_{j})-\frac{\tau}{2}\left[ \delta^{2}_{x}\phi(x_{j})-\phi(x_{j})- \phi(x_{j})\ln(\varepsilon^{2}+(\phi(x_{j}))^{2})\right],\\
\xi^{\varepsilon,n}_{j}:=&\delta_{t}^{2}u^{\varepsilon}(x_{j},t_{n})-\frac{1}{2}\delta_{x}^{2} \left( u^{\varepsilon}(x_{j},t_{n+1})+u^{\varepsilon}(x_{j},t_{n-1}) \right)+\frac{1}{2} \left( u^{\varepsilon}(x_{j},t_{n+1})+u^{\varepsilon}(x_{j},t_{n-1}) \right)\\
&+u^{\varepsilon}(x_{j},t_{n})f_{\varepsilon}\left( ( u^{\varepsilon}(x_{j},t_{n}))^{2} \right),~~j\in \mathcal{T}_{N},~~1\leq n\leq \frac{T}{\tau}-1,
\end{split}
\end{align}
then we have the following bounds for the local trunction error.
\begin{lem}\label{trunerrlem}
Under the assumption (A), we have
\begin{align}
&\|\xi^{\varepsilon,0}\|_{H^{1}}\lesssim h^{2}+\tau^{2},\\
&\|\xi^{\varepsilon, n}\|_{l^{2}}\lesssim h^{2}+\tau^{2},\\
&\|\delta^{+}_{x}\xi^{\varepsilon, n}\|_{l^{2}}\lesssim h^{2}+\tau^{2},~~1\leq n \leq \frac{T}{\tau}-1.
\end{align}
\end{lem}
\begin{proof}
By (\ref{initialvalue}) and the Taylor expansion, it leads to
\begin{align}
|\xi^{\varepsilon,0}_{j}|\leq\frac{\tau^{2}}{6}\|\partial_{t}^{3}u\|_{L^{\infty}(0,T;L^{\infty}(\Omega))}+\frac{\tau h}{6}\|\partial_{x}^{3}\phi\|_{L^{\infty}(\Omega)}\lesssim h^{2}+\tau^{2}.
\end{align}
Similarly, we have
\begin{align}
|\delta_{x}^{+}\xi^{\varepsilon,0}_{j}|\leq\frac{\tau^{2}}{6}\|\partial_{ttt}u\|_{L^{\infty}(0,T;H^{1}(\Omega))}+\frac{\tau h}{6}\|\partial_{x}^{4}\phi\|_{L^{\infty}(\Omega)}\lesssim h^{2}+\tau^{2},~~j\in \mathcal{T}_{N}.
\end{align}
Therefore
\begin{align}
\|\xi^{\varepsilon,0}\|_{H^{1}}\lesssim h^{2}+\tau^{2}.
\end{align}
Noting that
\begin{align}
\begin{split}
\xi^{\varepsilon,n}_{j}:=&\delta_{t}^{2}u^{\varepsilon}(x_{j},t_{n})-\frac{1}{2}\delta_{x}^{2} \left( u^{\varepsilon}(x_{j},t_{n+1})+u^{\varepsilon}(x_{j},t_{n-1}) \right)+\frac{1}{2} \left( u^{\varepsilon}(x_{j},t_{n+1})+u^{\varepsilon}(x_{j},t_{n-1}) \right)\\
&+ u^{\varepsilon}(x_{j},t_{n})f_{\varepsilon}\left( ( u^{\varepsilon}(x_{j},t_{n}))^{2} \right)\\
&-[ \partial_{tt}u^{\varepsilon}(x_{j},t_{n})-\partial_{xx} u^{\varepsilon}(x_{j},t_{n})+ u^{\varepsilon}(x_{j},t_{n})+ u^{\varepsilon}(x_{j},t_{n})f_{\varepsilon}\left( ( u^{\varepsilon}(x_{j},t_{n}))^{2} \right) ]\\
=&\left[ \delta_{t}^{2}u^{\varepsilon}(x_{j},t_{n})-\partial_{tt}u^{\varepsilon}(x_{j},t_{n})\right] - \left[\frac{1}{2}\delta_{x}^{2} \left( u^{\varepsilon}(x_{j},t_{n+1})+u^{\varepsilon}(x_{j},t_{n-1}) \right)-\partial_{xx}u^{\varepsilon}(x_{j},t_{n}) \right]\\
&+\frac{1}{2} \left( u^{\varepsilon}(x_{j},t_{n+1})+u^{\varepsilon}(x_{j},t_{n-1}) \right) -u^{\varepsilon}(x_{j},t_{n}).
\end{split}
\end{align}
Taking the Taylor expansion, we obtain
\begin{align}
\xi^{\varepsilon,n}_{j}=\frac{\tau^{2}}{12}\alpha^{\varepsilon,n}_{j}+\frac{\tau^{2}}{2}\beta^{\varepsilon,n}_{j}+\frac{h^{2}}{12}\eta^{\varepsilon,n}_{j}+\frac{\tau^{2}}{2}\phi^{\varepsilon,n}_{j},
\end{align}
where
\begin{align}
&\alpha^{\varepsilon,n}_{j}=\int^{1}_{-1}(1-|s|)^{3}\partial^{4}_{t} u^{\varepsilon}(x_{j},t_{n}+s\tau)ds,~~\beta^{\varepsilon,n}_{j}=\int^{1}_{-1}(1-|s|)\partial^{2}_{t} u^{\varepsilon}_{xx}(x_{j},t_{n}+s\tau)ds,\\
&\eta^{\varepsilon,n}_{j}=\int^{1}_{-1}(1-|s|)^{3}\left( \partial^{4}_{x} u^{\varepsilon}(x_{j}+sh,t_{n+1})+ \partial^{4}_{x} u^{\varepsilon}(x_{j}+sh,t_{n-1}) \right) ds,\\
&\phi^{\varepsilon,n}_{j}=\int^{1}_{-1}(1-|s|)\partial^{2}_{t} u^{\varepsilon}(x_{j},t_{n}+s\tau)ds.
\end{align}
Applying Cauchy-Schwarz inequality, we obtain
\begin{align}
\begin{split}
\|\alpha^{\varepsilon,n}\|^{2}_{l^{2}}&=h\sum\limits^{N-1}_{j=1}|\alpha^{\varepsilon,n}_{j}|^{2}\leq h\int^{1}_{-1}(1-|s|)^{6}ds\sum\limits^{N-1}_{j=1}\int^{1}_{-1}|\partial^{4}_{t} u^{\varepsilon}(x_{j},t_{n}+s\tau)|^{2}ds\\
&=\frac{2}{7}\left[ \int^{1}_{-1}\|\partial^{4}_{t} u^{\varepsilon}(\cdot,t_{n}+s\tau)\|^{2}_{L^{2}(\Omega)}ds- \int^{1}_{-1}\sum\limits^{N-1}_{j=0} \int^{x_{j+1}}_{x_{j}}\int^{\omega}_{x_{j}}\partial_{x} |\partial^{4}_{t} u^{\varepsilon}(\hat{x},t_{n}+s\tau)|^{2}d\hat{x}d\omega ds \right]\\
&\leq \frac{2}{7}\int^{1}_{-1}\left[ \|\partial^{4}_{t} u^{\varepsilon}(\cdot,t_{n}+s\tau)\|^{2}_{L^{2}(\Omega)} +2h\|\partial^{4}_{t} u^{\varepsilon}_{x}(\cdot,t_{n}+s\tau)\|_{L^{2}(\Omega)} \|\partial^{4}_{t} u^{\varepsilon}(\cdot,t_{n}+s\tau)\|_{L^{2}(\Omega)} \right] ds\\
&\leq \max_{0\leq t\leq T}\left( \|\partial^{4}_{t} u^{\varepsilon}\|_{L^{2}(\Omega)}+h \|\partial^{4}_{t} u_{x}^{\varepsilon}\|_{L^{2}(\Omega)} \right)^{2},
\end{split}
\end{align}
which implies that when $h\leq 1,$
\begin{align}
\left\|\alpha^{\varepsilon, n}\right\|_{l^{2}} \leq \left\|\partial_{t}^{4} u^{\varepsilon}\right\|_{L^{\infty}\left(0, T ; H^{1}(\Omega)\right)} .
\end{align}
Similarly, we can get
\begin{align}
&\|\beta^{\varepsilon, n}\|_{l^{2}} \leq 2\left\|\partial_{t}^{2} u^{\varepsilon}\right\|_{L^{\infty}\left(0, T ; H^{3}(\Omega)\right)},\\
&\left\|\phi^{\varepsilon, n}\right\|_{l^{2}} \leq \left\|\partial_{t}^{2} u^{\varepsilon}\right\|_{L^{\infty}\left(0, T ; H^{1}(\Omega)\right)} .
\end{align}
On the other hand, it can be estimated that
\begin{align}
\begin{split}
\|\eta^{\varepsilon, n}\|^{2}_{l^{2}} & \leq h \int_{-1}^{1}(1-|s|)^{6} d s \sum_{j=1}^{N-1} \int_{-1}^{1}\left|\partial_{x}^{4} u^{\varepsilon}\left(x_{j}+s h, t_{n+1}\right)+\partial_{x}^{4} u^{\varepsilon}\left(x_{j}+s h, t_{n-1}\right)\right|^{2} d s \\
& \leq \frac{4 h}{7} \sum_{j=1}^{N-1} \int_{-1}^{1}\left(\left|\partial_{x}^{4} u^{\varepsilon}\left(x_{j}+s h, t_{n+1}\right)\right|^{2}+\left|\partial_{x}^{4} u^{\varepsilon}\left(x_{j}+s h, t_{n-1}\right)\right|^{2}\right) d s \\
& \leq \frac{8}{7}\left(\left\|\partial_{x}^{4} u^{\varepsilon}\left(\cdot, t_{n-1}\right)\right\|_{L^{2}(\Omega)}^{2}+\left\|\partial_{x}^{4} u^{\varepsilon}\left(\cdot, t_{n+1}\right)\right\|_{L^{2}(\Omega)}^{2}\right) \\
& \leq 4\left\|u^{\varepsilon}\right\|_{L^{\infty}\left(0, T ; H^{4}(\Omega)\right)}^{2},
\end{split}
\end{align}
which yields $\|\eta^{\varepsilon, n}\|_{l^{2}}\leq 2\left\|u^{\varepsilon}\right\|_{L^{\infty}\left(0, T ; H^{4}(\Omega)\right)}$.
Therefore, according to the assumption (A), we get
\begin{align}
\begin{split}
\left\|\xi^{\varepsilon, n}\right\|_{l^{2}} \leq &\frac{\tau^{2}}{12}\left\|\partial_{t}^{4} u^{\varepsilon}\right\|_{L^{\infty}\left(0, T ; H^{1}(\Omega)\right)}+\tau^{2}\left\|\partial_{t}^{2} u^{\varepsilon}\right\|_{L^{\infty}\left(0, T ; H^{3}(\Omega)\right)}+\frac{h^{2}}{6}\left\|u^{\varepsilon}\right\|_{L^{\infty}\left(0, T ; H^{4}(\Omega)\right)}\\
&+\frac{\tau^{2}}{2}\left\|\partial_{t}^{2} u^{\varepsilon}\right\|_{L^{\infty}\left(0, T ; H^{1}(\Omega)\right)} \\
\lesssim &\tau^{2}+h^{2}.
\end{split}
\end{align}
Using the same approach, we can get
\begin{align}
\begin{split}
\|\delta^{+}_{x}\xi^{\varepsilon, n}\|_{l^{2}}\leq& \frac{\tau^{2}}{12}\left\|\partial_{t}^{4} u^{\varepsilon}\right\|_{L^{\infty}\left(0, T ; H^{2}(\Omega)\right)}+\tau^{2}\left\|\partial_{t}^{2} u^{\varepsilon}\right\|_{L^{\infty}\left(0, T ; H^{4}(\Omega)\right)}+\frac{h^{2}}{6}\left\|u^{\varepsilon}\right\|_{L^{\infty}\left(0, T ; H^{5}(\Omega)\right)}\\
&+\frac{\tau^{2}}{2}\left\|\partial_{t}^{2} u^{\varepsilon}\right\|_{L^{\infty}\left(0, T ; H^{2}(\Omega)\right)}  \\
\lesssim & \tau^{2}+h^{2}.
\end{split}
\end{align}
This ends the proof.
\end{proof}


Subtracting (\ref{SIFD2}) from  (\ref{trunerre}), the error $e^{\varepsilon,n}_{j}$ satisfies
\begin{subequations}\label{SIFD2error}
\begin{align}
&\delta_{t}^{2}e^{\varepsilon,n}_{j}-\frac{1}{2}\delta_{x}^{2} (e^{\varepsilon,n+1}_{j}+e^{\varepsilon,n-1}_{j})+\frac{1}{2} (e^{\varepsilon,n+1}_{j}+e^{\varepsilon,n-1}_{j})=\xi^{\varepsilon,n}_{j} - \zeta^{\varepsilon,n}_{j},\label{SIFD2errora}\\
&e^{\varepsilon,n}_{0}=e^{\varepsilon,n}_{N},~~e^{\varepsilon,n}_{-1}=e^{\varepsilon,n}_{N-1},~~n=0,1,\ldots\label{SIFD2errorb}\\
&e^{\varepsilon,0}_{j}=0,~~e^{\varepsilon,1}_{j}=\tau\xi^{0}_{j},~~j\in \mathcal{T}_{N},\label{SIFD2errorc}
\end{align}
\end{subequations}
where
\begin{align}
\zeta^{\varepsilon,n}_{j}=u^{\varepsilon}(x_{j},t_{n})f_{\varepsilon}\left( ( u^{\varepsilon}(x_{j},t_{n}))^{2} \right)-u^{\varepsilon,n}_{j}f_{\varepsilon}\left( ( u^{\varepsilon,n}_{j})^{2} \right).
\end{align}
We define the ``energy'' for the error vector $e^{\varepsilon,n}~(n=0,1,\ldots)$ as
\begin{align}
E^{n}_{e}:=&\|\delta_{t}^{+}e^{\varepsilon,n}\|^{2}_{l^{2}}+\frac{1}{2}\left( \|\delta_{x}^{+}e^{\varepsilon,n+1}\|^{2}_{l^{2}}+ \|\delta_{x}^{+}e^{\varepsilon,n}\|^{2}_{l^{2}} \right)+\frac{1}{2}\left( \|e^{\varepsilon,n+1}\|^{2}_{l^{2}}+\|e^{\varepsilon,n}\|^{2}_{l^{2}} \right).
\end{align}
Besides, we can get that
\begin{align}\label{E0}
E^{0}_{e}:=&\|\xi^{0}\|^{2}_{l^{2}}+\frac{\tau^{2}}{2}\|\delta_{x}^{+}\xi^{0}\|^{2}_{l^{2}}+\frac{\tau^{2}}{2} \|\xi^{0}\|^{2}_{l^{2}}\lesssim (\tau^{2}+h^{2})^{2}.
\end{align}
\begin{proof} ($\mathbf{Proof~of~Theorem~{\ref{SIFD2err}}}$)
When $k=1$,
under the assumption (A), by Lemma \ref{trunerrlem} we can conclude the errors of the first step discretization (\ref{initialvalue})
\begin{align}
e^{\varepsilon,0}=0,~~\|e^{\varepsilon,1}_{j}\|_{H^{1}}\lesssim\tau^{2}+h^{2},
\end{align}
for sufficiently small $0<\tau<\tau_{1}$ and $0<h<h_{1}$. So it is true for $k=0, 1$.

Assume (\ref{SIFD2error}) is valid for $k\leq n\leq \frac{T}{\tau}-1$. Next, we need to verify (\ref{SIFD2error}) is true for $k=n+1$.
Denote
\begin{align}
\zeta^{\varepsilon,m}_{j}=u^{\varepsilon}(x_{j},t_{m})f_{\varepsilon}\left( ( u^{\varepsilon}(x_{j},t_{m}))^{2} \right)-u^{\varepsilon,m}_{j}f_{\varepsilon}\left( ( u^{\varepsilon,m}_{j})^{2} \right).
\end{align}
When $|u^{\varepsilon,m}_{j}|\leq|u^{\varepsilon}(x_{j},t_{m})|$, we get
\begin{align}
\begin{split}
|\zeta^{\varepsilon,m}_{j}|=&|u^{\varepsilon,m}_{j}\ln \left( {\varepsilon^{2}}+( u^{\varepsilon,m}_{j} )^{2} \right)-u^{\varepsilon,m}_{j}\ln \left( {\varepsilon^{2}}+( u^{\varepsilon}(x_{j},t_{m}))^{2} \right) \\
&+ u^{\varepsilon,m}_{j}\ln \left( {\varepsilon^{2}}+( u^{\varepsilon}(x_{j},t_{m}))^{2} \right)
-u^{\varepsilon}(x_{j},t_{m})\ln \left( {\varepsilon^{2}}+( u^{\varepsilon}(x_{j},t_{m}))^{2} \right) |\\
=&\left| e^{\varepsilon,m}_{j}\ln \left( {\varepsilon^{2}}+( u^{\varepsilon}(x_{j},t_{m}))^{2} \right)+ u^{\varepsilon,m}_{j} \ln \left(1+ \frac{ ( u^{\varepsilon,m}_{j} )^{2}-(u^{\varepsilon}(x_{j},t_{m}))^{2} }{{\varepsilon^{2}}+( u^{\varepsilon}(x_{j},t_{m}) )^{2}} \right) \right|\\
\leq & \left| e^{\varepsilon,m}_{j}\ln \left( {\varepsilon^{2}}+( u^{\varepsilon}(x_{j},t_{m}))^{2} \right)\right|+ \left|u^{\varepsilon}(x_{j},t_{m}) \frac{ ( u^{\varepsilon,m}_{j} )^{2}- \left(u^{\varepsilon}(x_{j},t_{m})\right)^{2} }{( u^{\varepsilon}(x_{j},t_{m}) )^{2}}\right|\\
\leq & | e^{\varepsilon,m}_{j}|\max\{\ln(\frac{1}{\varepsilon^{2}}),|\ln(\Lambda^{2}+{\varepsilon^{2}})|\}+2| e^{\varepsilon,m}_{j}|\\
=&| e^{\varepsilon,m}_{j}|\left( \max\{\ln(\frac{1}{\varepsilon^{2}}),|\ln(\Lambda^{2}+{\varepsilon^{2}})|\}+2\right).
\end{split}
\end{align}
In addition, when $|u^{\varepsilon,m}_{j}|\geq|u^{\varepsilon}(x_{j},t_{m})|$, we obtain
\begin{align}
\begin{split}
|\zeta^{\varepsilon,m}_{j}|=&|u^{\varepsilon}(x_{j},t_{m})\ln \left( {\varepsilon^{2}}+( u^{\varepsilon}(x_{j},t_{m}))^{2} \right)-
u^{\varepsilon}(x_{j},t_{m})\ln \left( {\varepsilon^{2}}+( u^{\varepsilon,m}_{j} )^{2} \right)\\
&+u^{\varepsilon}(x_{j},t_{m})\ln \left( {\varepsilon^{2}}+( u^{\varepsilon,m}_{j} )^{2} \right)
- u^{\varepsilon,m}_{j}\ln \left( {\varepsilon^{2}}+( u^{\varepsilon,m}_{j} )^{2} \right) |\\
=&\left| e^{\varepsilon,m}_{j}\ln \left( {\varepsilon^{2}}+( u^{\varepsilon,m}_{j})^{2} \right)+ u^{\varepsilon}(x_{j},t_{m}) \ln \left(\frac{ {\varepsilon^{2}}+( u^{\varepsilon}(x_{j},t_{m}))^{2} }{{\varepsilon^{2}}+( u^{\varepsilon,m}_{j} )^{2}} \right) \right|\\
\leq & | e^{\varepsilon,m}_{j}|\max\{\ln(\frac{1}{\varepsilon^{2}}),|\ln((1+\Lambda)^{2}+{\varepsilon^{2}})|\}+2| e^{\varepsilon,m}_{j}|\\
=&| e^{\varepsilon,m}_{j}|\left( \max\{\ln(\frac{1}{\varepsilon^{2}}),|\ln((1+\Lambda)^{2}+{\varepsilon^{2}})|\}+2\right),
\end{split}
\end{align}
where we use the assumption $\|u^{\varepsilon,m}\|_{l^{\infty}}\leq 1+\Lambda$ above for $m\leq n$. Since $\varepsilon$ is sufficiently small, we have
\begin{align}\label{nonlinearerror}
\|\zeta^{\varepsilon,m}\|^{2}_{l^{2}}\lesssim \| e^{\varepsilon,m}\|^{2}_{l^{2}}(\ln\varepsilon^{2})^{2}.
\end{align}

Multiplying both sides of (\ref{SIFD2errora}) by $h(e^{\varepsilon,m+1}-e^{\varepsilon,m-1})$, then summing up for $j\in \mathcal{T}_{N}$. And by Young's inequality, Lemma \ref{trunerrlem}, and (\ref{nonlinearerror}) we can obtain
\begin{align}
\begin{split}
E^{m}_{e}-E^{m-1}_{e}&=h\sum\limits_{j=0}^{N-1}(\xi^{\varepsilon,m}_{j}-\zeta^{\varepsilon,m}_{j})(e^{\varepsilon,m+1}_{j}-e^{\varepsilon,m-1}_{j})\\
&\leq h\sum\limits_{j=0}^{N-1}\left( |\xi^{\varepsilon,m}_{j}|+|\zeta^{\varepsilon,m}_{j}|\right)\left|e^{\varepsilon,m+1}_{j}-e^{\varepsilon,m-1}_{j} \right|\\
&\leq\tau \left( \|\xi^{\varepsilon,m}\|^{2}_{l^{2}}+\|\zeta^{\varepsilon,m}\|^{2}_{l^{2}}+\left\|\delta^{+}_{t}e^{\varepsilon,m}\|^{2}_{l^{2}}+\|\delta^{+}_{t}e^{\varepsilon,m-1} \right\|^{2}_{l^{2}}\right)\\
&\lesssim \tau\left((\tau^{2}+h^{2})^{2} + (\ln\varepsilon^{2})^{2}\| e^{\varepsilon,m}\|^{2}_{l^{2}}+E^{m}_{e}+E^{m-1}_{e}\right),~~1\leq m\leq \frac{T}{\tau}-1.
\end{split}
\end{align}
 Therefore, there exists a constant $\tau_{2}>0$ sufficiently small and independent of $\varepsilon$ and $h$, such that when $0<\tau<\tau_{2}$, we get
 \begin{align}
 E^{m}_{e}-E^{m-1}_{e}
 &\lesssim \tau\left((\tau^{2}+h^{2})^{2} + (\ln\varepsilon^{2})^{2}\| e^{\varepsilon,m}\|^{2}_{l^{2}}+E^{m-1}_{e}\right),~~1\leq m\leq
 \frac{T}{\tau}-1.
 \end{align}
Summing above the inequalities up to $n$, and noticing (\ref{E0}), the following holds
\begin{align}
&E^{n}_{e}\lesssim   (\tau^{2}+h^{2})^{2}+ \tau(\ln\varepsilon^{2})^{2}\sum\limits_{m=0}^{n-1}E^{m}_{e},~~1\leq n\leq \frac{T}{\tau}-1.
\end{align}
By applying the discrete Gronwall's inequality \cite{holte2009discrete}, we have
\begin{align}
&E^{n}_{e}\lesssim e^{T(\ln(\varepsilon^{2}))^{2}}(\tau^{2}+h^{2})^{2},~~1\leq n\leq \frac{T}{\tau}-1.
\end{align}
Recalling $\|\delta_{x}^{+}e^{\varepsilon,n+1}\|^{2}_{l^{2}} + \|e^{\varepsilon,n+1}\|^{2}_{l^{2}}\leq 2E^{n}_{e}$ when $0<\varepsilon \ll1$, we can get the error estimate
\begin{align}
\|\delta_{x}^{+}e^{\varepsilon,n+1}\|^{2}_{l^{2}} + \|e^{\varepsilon,n+1}\|^{2}_{l^{2}}\lesssim 
&e^{T(\ln(\varepsilon^{2}))^{2}}(\tau^{2}+h^{2})^{2}.
\end{align}
By Sobolev inequality, we obtain
\begin{align}
\|e^{\varepsilon,n}\|^{2}_{l^{\infty}}\leq\|\delta_{x}^{+}e^{\varepsilon,n}\|^{2}_{l^{2}} + \|e^{\varepsilon,n}\|^{2}_{l^{2}}\lesssim 
&e^{T(\ln(\varepsilon^{2}))^{2}}(\tau^{2}+h^{2})^{2}.
\end{align}
Therefore, there exist $\tau_{3}>0,h_{2}>0$ sufficiently small. When $0<h<h_{2},0<\tau<\tau_{3}$, applying the triangle inequality, it implies that
\begin{align}
\|u^{\varepsilon,n}\|_{l^{\infty}}\leq \|u^{\varepsilon}(\cdot,t_{n})\|_{L^{\infty}(\Omega)}+\|e^{\varepsilon,n}\|_{l^{\infty}}\leq \Lambda +1.
\end{align}
We complete the proof by choosing $h_{0}=\min\{h_{1},h_{2}\},~\tau_{0}=\min\{\tau_{1},\tau_{2},\tau_{3}\}$.
\end{proof}

\subsection{The proof of Theorem for EFD}

Define the local trunction error for the EFD (\ref{EFD}) as
\begin{align}\label{EFDtrunerre}
\begin{split}
\xi^{0}_{j}:=&\delta_{t}^{+}u^{\varepsilon}(x_{j},0)-\gamma(x_{j})-\frac{\tau}{2}\left[ \delta^{2}_{x}\phi(x_{j})-\phi(x_{j})- \phi(x_{j})\ln(\varepsilon^{2}+(\phi(x_{j}))^{2})\right],~~j\in \mathcal{T}_{N},\\
\xi^{\varepsilon,n}_{j}:=&\delta_{t}^{2}u^{\varepsilon}(x_{j},t_{n})-\delta_{x}^{2} u^{\varepsilon}(x_{j},t_{n})+ u^{\varepsilon}(x_{j},t_{n})+ u^{\varepsilon}(x_{j},t_{n})f_{\varepsilon}\left( ( u^{\varepsilon}(x_{j},t_{n}))^{2} \right),~~1\leq n\leq \frac{T}{\tau}-1,
\end{split}
\end{align}
then we have the following bounds for the local trunction error.
\begin{lem}\label{EFDtrunerrlem}
Under the assumption (A), we have
\begin{align}
&\|\xi^{\varepsilon,0}\|_{H^{1}}\lesssim h^{2}+\tau^{2},\\
&\|\xi^{\varepsilon, n}\|_{l^{2}}\lesssim h^{2}+\tau^{2},\\
&\|\delta^{+}_{x}\xi^{\varepsilon, n}\|_{l^{2}}\lesssim h^{2}+\tau^{2},~~1\leq n \leq \frac{T}{\tau}-1.
\end{align}
\end{lem}
\begin{proof}
According to the Lemma \ref{trunerrlem}, we have
\begin{align}
\|\xi^{\varepsilon,0}\|_{H^{1}}\lesssim h^{2}+\tau^{2}.
\end{align}
Noting that
\begin{align}
\begin{split}
\xi^{\varepsilon,n}_{j}:=&\delta_{t}^{2}u^{\varepsilon}(x_{j},t_{n})-\delta_{x}^{2} u^{\varepsilon}(x_{j},t_{n})+ u^{\varepsilon}(x_{j},t_{n})+ u^{\varepsilon}(x_{j},t_{n})f_{\varepsilon}\left( ( u^{\varepsilon}(x_{j},t_{n}))^{2} \right)\\
&-[ \partial_{tt}u^{\varepsilon}(x_{j},t_{n})-\partial_{xx} u^{\varepsilon}(x_{j},t_{n}) + u^{\varepsilon}(x_{j},t_{n})+ u^{\varepsilon}(x_{j},t_{n})f_{\varepsilon}\left( ( u^{\varepsilon}(x_{j},t_{n}))^{2} \right) ]\\
=&\left[ \delta_{t}^{2}u^{\varepsilon}(x_{j},t_{n})-\partial_{tt}u^{\varepsilon}(x_{j},t_{n})\right] - \left[\delta_{x}^{2} u^{\varepsilon}(x_{j},t_{n})-\partial_{xx} u^{\varepsilon}(x_{j},t_{n})\right].
\end{split}
\end{align}
Taking the Taylor expansion, we obtain
\begin{align}
\xi^{\varepsilon,n}_{j}=\frac{\tau^{2}}{12}\alpha^{\varepsilon,n}_{j}+\frac{h^{2}}{6}\beta^{\varepsilon,n}_{j},
\end{align}
where
\begin{align}
&\alpha^{\varepsilon,n}_{j}=\int^{1}_{-1}(1-|s|)^{3}\partial^{4}_{t} u^{\varepsilon}(x_{j},t_{n}+s\tau)ds,~~\beta^{\varepsilon,n}_{j}=\int^{1}_{-1}(1-|s|)^{3}\partial^{3}_{x} u^{\varepsilon}(x_{j}+sh,t_{n})ds.
\end{align}
By Cauchy-Schwarz inequality, we obtain
\begin{align}
\begin{split}
\|\alpha^{\varepsilon,n}\|^{2}_{l^{2}}&=h\sum\limits^{N-1}_{j=1}|\alpha^{\varepsilon,n}_{j}|^{2}\leq h\int^{1}_{-1}(1-|s|)^{6}ds\sum\limits^{N-1}_{j=1}\int^{1}_{-1}|\partial^{4}_{t} u^{\varepsilon}(x_{j},t_{n}+s\tau)|^{2}ds\\
&=\frac{2}{7}\left[ \int^{1}_{-1}\|\partial^{4}_{t} u^{\varepsilon}(\cdot,t_{n}+s\tau)\|^{2}_{L^{2}(\Omega)}ds- \int^{1}_{-1}\sum\limits^{N-1}_{j=0} \int^{x_{j+1}}_{x_{j}}\int^{\omega}_{x_{j}}\partial_{x} |\partial^{4}_{t} u^{\varepsilon}(\hat{x},t_{n}+s\tau)|^{2}d\hat{x}d\omega ds \right]\\
&\leq \frac{2}{7}\int^{1}_{-1}\left[ \|\partial^{4}_{t} u^{\varepsilon}(\cdot,t_{n}+s\tau)\|^{2}_{L^{2}(\Omega)} +2h\|\partial^{4}_{t}\partial_{x} u^{\varepsilon}(\cdot,t_{n}+s\tau)\|_{L^{2}(\Omega)} \|\partial^{4}_{t} u^{\varepsilon}(\cdot,t_{n}+s\tau)\|_{L^{2}(\Omega)} \right] ds\\
&\leq \max_{0\leq t\leq T}\left( \|\partial^{4}_{t} u^{\varepsilon}\|_{L^{2}(\Omega)}+h \|\partial^{4}_{t}\partial_{x} u^{\varepsilon}\|_{L^{2}(\Omega)} \right)^{2},
\end{split}
\end{align}
which implies that when $h\leq 1,$
\begin{align}
\left\|\alpha^{\varepsilon, n}\right\|_{l^{2}} \leq \left\|\partial_{t}^{4} u^{\varepsilon}\right\|_{L^{\infty}\left(0, T ; H^{1}(\Omega)\right)}.
\end{align}
On the other hand, it can be estimated that
\begin{align}
\begin{split}
\|\beta^{\varepsilon, n}\|^{2}_{l^{2}} & \leq h \int_{-1}^{1}(1-|s|)^{6} d s \sum_{j=1}^{N-1} \int_{-1}^{1}\left|\partial_{x}^{3} u^{\varepsilon}\left(x_{j}+s h, t_{n}\right)\right|^{2} d s \\
& \leq \frac{2 h}{7} \sum_{j=1}^{N-1} \int_{-1}^{1}\left|\partial_{x}^{3} u^{\varepsilon}\left(x_{j}+s h, t_{n}\right)\right|^{2} d s \\
& \leq \frac{4}{7}\left(\left\|\partial_{x}^{3} u^{\varepsilon}\left(\cdot, t_{n}\right)\right\|_{L^{2}(\Omega)}^{2}\right) \\
& \leq \left\|u^{\varepsilon}\right\|_{L^{\infty}\left(0, T ; H^{3}(\Omega)\right)}^{2},
\end{split}
\end{align}
which yields that $\|\beta^{\varepsilon, n}\|_{l^{2}}\leq \left\|u^{\varepsilon}\right\|_{L^{\infty}\left(0, T ; H^{3}(\Omega)\right)}$.
Therefore, according to assumption (A), we get
\begin{align}
\begin{split}
\left\|\xi^{\varepsilon, n}\right\|_{l^{2}} & \leq \frac{\tau^{2}}{12}\left\|\partial_{t}^{4} u^{\varepsilon}\right\|_{L^{\infty}\left(0, T ; H^{1}(\Omega)\right)}+\frac{h^{2}}{6}\left\|u^{\varepsilon}\right\|_{L^{\infty}\left(0, T ; H^{3}(\Omega)\right)} \\
& \lesssim \tau^{2}+h^{2}.
\end{split}
\end{align}
With the same method, we have
\begin{align}
\begin{split}
\|\delta^{+}_{x}\xi^{\varepsilon, n}\|_{l^{2}}& \leq \frac{\tau^{2}}{12}\left\|\partial_{t}^{4} u^{\varepsilon}\right\|_{L^{\infty}\left(0, T ; H^{2}(\Omega)\right)}+\frac{h^{2}}{6}\left\|u^{\varepsilon}\right\|_{L^{\infty}\left(0, T ; H^{4}(\Omega)\right)} \\
& \lesssim \tau^{2}+h^{2}.
\end{split}
\end{align}
This completes the proof.
\end{proof}

Subtracting (\ref{EFD}) from  (\ref{EFDtrunerre}), the error $e^{\varepsilon,n}_{j}$ satisfies
\begin{subequations}\label{EFDerror}
\begin{align}
&\delta_{t}^{2}e^{\varepsilon,n}_{j}-\delta_{x}^{2} e^{\varepsilon,n}_{j}+ e^{\varepsilon,n}_{j}=\xi^{\varepsilon,n}_{j} - \zeta^{\varepsilon,n}_{j},\label{EFDerrora}\\
&e^{\varepsilon,n}_{0}=e^{\varepsilon,n}_{N},~~e^{\varepsilon,n}_{-1}=e^{\varepsilon,n}_{N-1},~~n=0,1,\ldots\label{EFDerrorb}\\
&e^{\varepsilon,0}_{j}=0,~~e^{\varepsilon,1}_{j}=\tau\xi^{0}_{j},~~j\in \mathcal{T}_{N},\label{EFDerrorc}
\end{align}
\end{subequations}
where
\begin{align}
\zeta^{\varepsilon,n}_{j}=u^{\varepsilon}(x_{j},t_{n})f_{\varepsilon}\left( ( u^{\varepsilon}(x_{j},t_{n}))^{2} \right)-u^{\varepsilon,n}_{j}f_{\varepsilon}\left( ( u^{\varepsilon,n}_{j})^{2} \right).
\end{align}
We define the ``energy'' for the error vector $e^{\varepsilon,n}~(n=0,1,\ldots)$ as
\begin{align}
E^{n}_{e}:=&(1-\frac{\tau^{2}}{2}-\frac{\tau^{2}}{h^{2}})\|\delta_{t}^{+}e^{\varepsilon,n}\|^{2}_{l^{2}}+\frac{1}{2}\left( \|e^{\varepsilon,n}\|^{2}_{l^{2}}+\|e^{\varepsilon,n+1}\|^{2}_{l^{2}} \right)+ \frac{1}{2h}\sum\limits_{j=0}^{N-1}\left[ (e^{\varepsilon,n+1}_{j+1}-e^{\varepsilon,n}_{j})^{2} + (e^{\varepsilon,n}_{j+1}-e^{\varepsilon,n+1}_{j})^{2}\right].
\end{align}
Besides, we can get that
\begin{align}\label{EFDE0}
 E^{0}_{e}:=&(1-\frac{\tau^{2}}{2}-\frac{\tau^{2}}{h^{2}})\|\delta_{t}^{+}e^{\varepsilon,0}\|^{2}_{l^{2}}+(\frac{1}{2}+\frac{1}{h^{2}}) \|e^{\varepsilon,1}\|^{2}_{l^{2}}=\|\xi^{0}\|^{2}_{l^{2}}\lesssim (\tau^{2}+h^{2})^{2}
 .
\end{align}
\begin{proof} ($\mathbf{Proof~of~Theorem~{\ref{EFDerr}}}$)
When $m=1$,
under assumption (A), by Lemma \ref{trunerrlem} we can conclude the first step errors of the discretization (\ref{initialvalue})
\begin{align}
e^{\varepsilon,0}=0,~~\|e^{\varepsilon,1}_{j}\|_{H^{1}}\lesssim\tau^{2}+h^{2},
\end{align}
for sufficiently small $0<\tau<\tau_{1}$ and $0<h<h_{1}$. So it is true for $m=0, 1$.
Assume (\ref{SIFD2error}) is valid for $m\leq n\leq \frac{T}{\tau}-1$. Next, we need to verify (\ref{SIFD2error}) is true for $m=n+1$.
Denote
\begin{align}
\zeta^{\varepsilon,m}_{j}=u^{\varepsilon}(x_{j},t_{m})f_{\varepsilon}\left( ( u^{\varepsilon}(x_{j},t_{m}))^{2} \right)-u^{\varepsilon,m}_{j}f_{\varepsilon}\left( ( u^{\varepsilon,m}_{j})^{2} \right).
\end{align}
With the same method in Theorem \ref{SIFD2err}, we have
\begin{align}\label{EFDnonlinearerror}
\|\zeta^{\varepsilon,m}\|^{2}_{l^{2}}\lesssim \| e^{\varepsilon,m}\|^{2}_{l^{2}}(\ln\varepsilon^{2})^{2}.
\end{align}
Besides, under the assumption $\tau \leq \frac{1}{2}\min\{1,h\}$, we have $1-\frac{\tau^{2}}{2}-\frac{\tau^{2}}{h^{2}}\geq \frac{1}{4}>0$.
By
\begin{align}
\left\|\delta_{x}^{+} e^{m+1}\right\|_{l^{2}}^{2}=\frac{1}{h} \sum_{j=0}^{N-1}\left(e_{j+1}^{m+1}-e_{j}^{m}-\tau \delta_{t}^{+} e_{j}^{m}\right)^{2} \leq \frac{2}{h} \sum_{j=0}^{N-1}\left(e_{j+1}^{m+1}-e_{j}^{m}\right)^{2}+\frac{2 \tau^{2}}{h^{2}}\left\|\delta_{t}^{+} e^{m}\right\|_{l^{2}}^{2},
\end{align}
we have
\begin{align}
E^{m}_{e} \geq \frac{1}{4}\left\|\delta_{x}^{+} e^{m+1}\right\|_{l^{2}}^{2}+\frac{1}{2}\left(\left\|e^{m}\right\|_{l^{2}}^{2}+\left\|e^{m+1}\right\|_{l^{2}}^{2}\right), \quad 1 \leq m \leq n-1.
\end{align}
Similar to the proof of Theorem \ref{SIFD2err}, there exists a $\tau_{2}>0$ sufficiently small, when $0<\tau\leq\tau_{2}$, we get
\begin{align}
E^{m}_{e}\lesssim 
&e^{T(\ln(\varepsilon^{2}))^{2}}(\tau^{2}+h^{2})^{2}, \quad 1 \leq m \leq n-1.
\end{align}
Therefore, we can get the $(n+1)th$ error estimate
\begin{align}
\|\delta_{x}^{+}e^{\varepsilon,n+1}\|^{2}_{l^{2}} + \|e^{\varepsilon,n+1}\|^{2}_{l^{2}}\lesssim 
&e^{T(\ln(\varepsilon^{2}))^{2}}(\tau^{2}+h^{2})^{2}.
\end{align}
By Sobolev inequality, we obtain
\begin{align}
\|e^{\varepsilon,n}\|^{2}_{l^{\infty}}\leq\|\delta_{x}^{+}e^{\varepsilon,n}\|^{2}_{l^{2}} + \|e^{\varepsilon,n}\|^{2}_{l^{2}}\lesssim 
&e^{T(\ln(\varepsilon^{2}))^{2}}(\tau^{2}+h^{2})^{2}.
\end{align}
Applying the triangle inequality, it implies that
\begin{align}
\|u^{\varepsilon,n}\|_{l^{\infty}}\leq \|u^{\varepsilon}(\cdot,t_{n})\|_{L^{\infty}(\Omega)}+\|e^{\varepsilon,n}\|_{l^{\infty}}\leq \Lambda +1.
\end{align}
This ends the proof  by choosing $h_{0}=h_{1},\tau_{0}=\min\{\tau_{1},\tau_{2}\}$.
\end{proof}
\section{Numerical results}\label{conclusion}
In this section, we represent some numerical experiments of the EFD (\ref{EFD}) scheme to quantify the error bounds. Since the results of the SIFD (\ref{SIFD2}) are similar to the EFD (\ref{EFD}), we omit the details here for brevity. Here we take $d=1, \lambda=1$ and we define the error functions as:
\begin{align}
&\hat{e}^{\varepsilon}(t_{n}):=u(\cdot,t_{n})-u^{\varepsilon}(\cdot,t_{n}),
\quad e^{\varepsilon}(t_{n}):=u^{\varepsilon}(\cdot,t_{n})-u^{\varepsilon,n}, \quad \tilde{e}^{\varepsilon}(t_{n}):=u(\cdot,t_{n})-u^{\varepsilon,n}.
\end{align}
Besides, we denote the error functions:
\begin{align}
&e^{\varepsilon}_{\infty}(t_{n}):=\|u^{\varepsilon}(\cdot,t_{n})-u^{\varepsilon,n}\|_{l^{\infty}}, \quad  e^{\varepsilon}_{2}(t_{n}):=\|u^{\varepsilon}(\cdot,t_{n})-u^{\varepsilon,n}\|_{l^{2}}, \\
&e^{\varepsilon}_{H^{1}}(t_{n}):=\sqrt{(e^{\varepsilon}_{2}(t_{n}))^{2}+\|\delta^{+}_{x}(u^{\varepsilon}(\cdot,t_{n})-u^{\varepsilon,n})\|^{2}_{l^{2}}}.
\end{align}
Here $u, u^{\varepsilon}$ are the exact solutions of the LogKGE (\ref{LogKGE}) and the RLogKGE (\ref{RLogKGE}), $u^{n}, u^{\varepsilon,n}$ are the numerical solutions of the LogKGE (\ref{LogKGE}) and the RLogKGE (\ref{RLogKGE}).

$\mathbf{Example~1}$. The initial datum is taken as $\phi(x)=e^{-\frac{k^{2}x^{2}}{2(c^{2}-k^{2})}},\gamma(x)=\frac{ckx}{c^2-k^2}e^{-\frac{(kx)^{2}}{2(c^{2}-k^{2})}}$, and the Gaussian solitary wave solution is
\begin{align}
u(x,t)=e^{-\frac{(kx-ct)^{2}}{2(c^{2}-k^{2})}},
\end{align}
where $c = 2, k=1$. The RLogKGE (\ref{RLogKGE}) is simulated on the domain $\Omega=[-16,16]$. The `exact' solution $u^{\varepsilon}$ is obtained numerically by the EFD (\ref{EFD}) scheme with $\varepsilon=10^{-7}$.

$\mathbf{Example~2}$. We take the initial value as $\phi(x)=\frac{2}{e^{-x^{2}}+e^{x^{2}}},\gamma(x)=0$.  The computation domain is chosen as $\Omega=[-16,16]$ with periodic boundary conditions. Since the analytical solution is not available in this example. The `exact' solution $u^{\varepsilon}$ is obtained by the EFD (\ref{EFD}) with a small mesh size $h=2^{-10}$, and time step $\tau=0.01\times2^{-9}$. In addition, the `exact' solution $u$ is approximated by $u^{\varepsilon}$ with $\varepsilon=10^{-7}$.
\subsection{Convergence of the regularized model}
Here we test the order of accuracy of the regularized model, that is the convergence rate between the solutions of the RLogKGE (\ref{RLogKGE}) and the LogKGE (\ref{LogKGE}). Figure \ref{fig:regexhat} represents $\|\hat{e}^{\varepsilon}\|_{l^{2}}, \|\hat{e}^{\varepsilon}\|_{l^{\infty}},\|\hat{e}^{\varepsilon}\|_{H^{1}}$ with the scheme EFD (\ref{EFD}) for Example~1 and Example~2. The errors are displayed at $T=0.5$.

From Figure \ref{fig:regexhat}, we can observe that the solutions of the RLogKGE (\ref{RLogKGE}) are linearly convergent to the LogKGE (\ref{LogKGE}) with regard to $\varepsilon$, and the convergence rate is $O(\varepsilon)$ in the $l^{2}$-norm, $l^{\infty}$-norm, $H^{1}$-norm.
\begin{center}
\begin{figure}[htbp]
\centering
\subfigure{\includegraphics[width=.43\textwidth, height=0.38\textwidth]{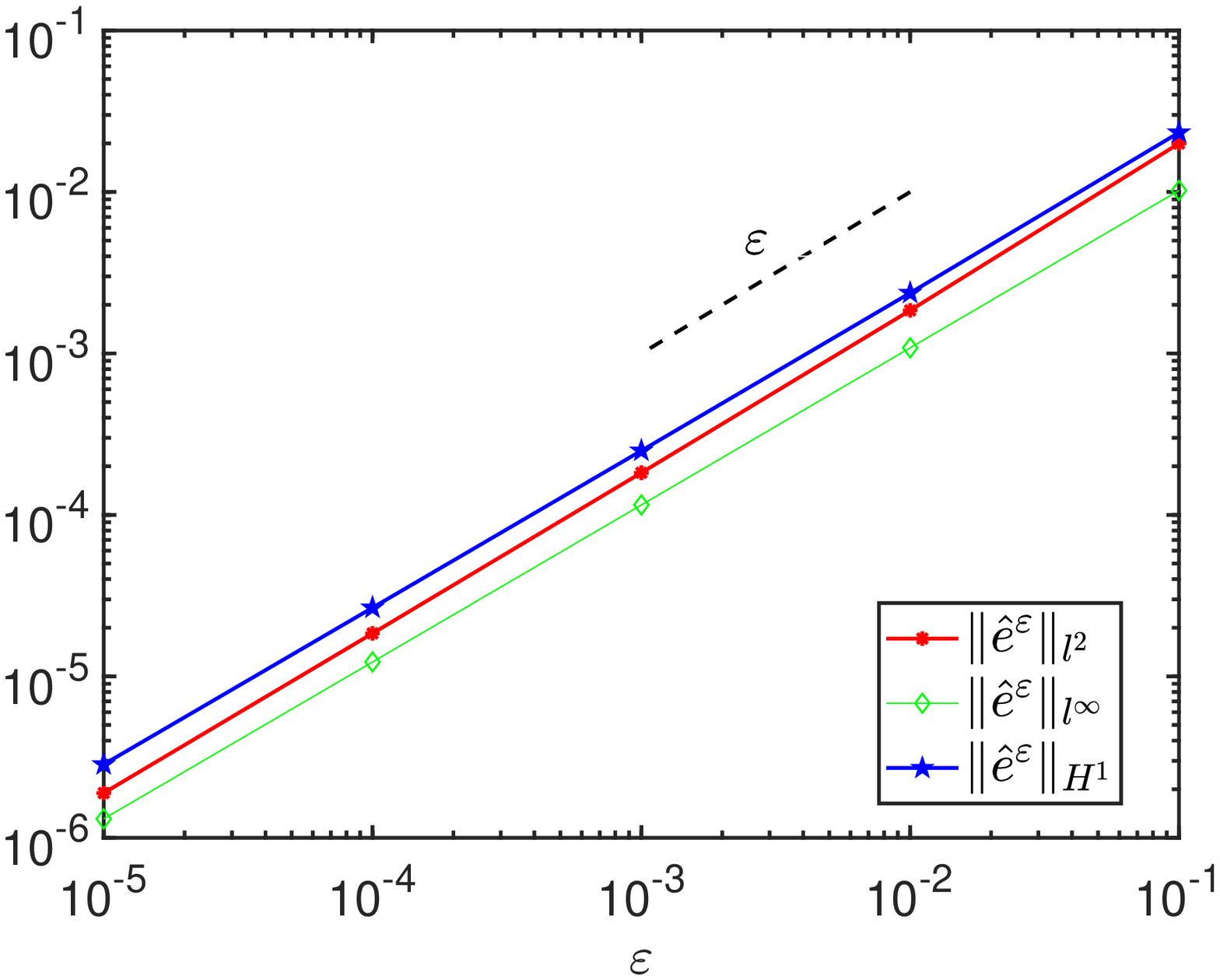}}~~~~~~~
\subfigure{\includegraphics[width=.43\textwidth, height=0.38\textwidth]{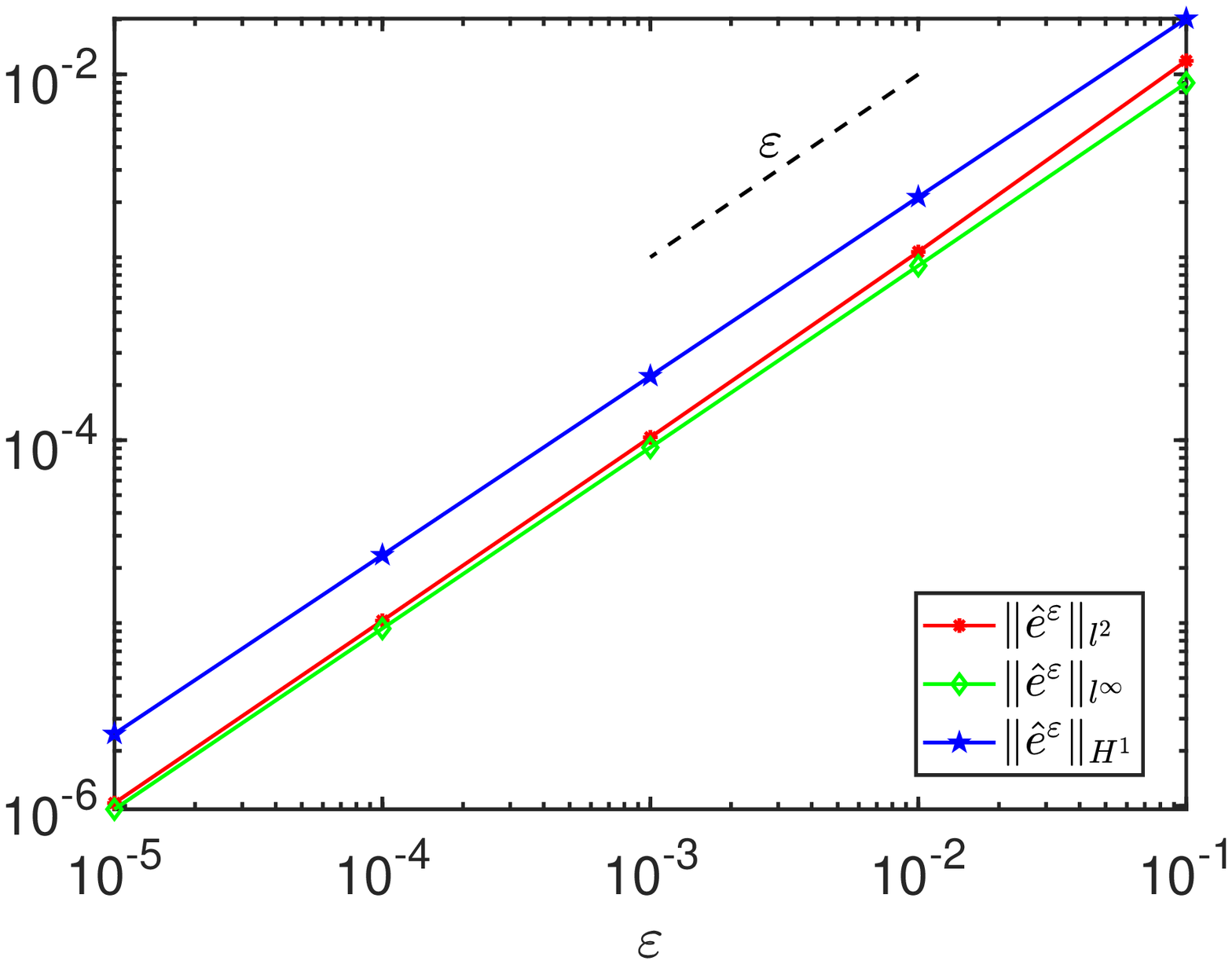}}
\caption{The errors $\hat{e}^{\varepsilon}(0.5)$ in three different norms with the scheme  EFD $(\ref{EFD})$ for Example~1 (left) and Example~2 (right).}
\label{fig:regexhat}
\end{figure}
\end{center}
\subsection{Convergence of FDTD to the RLogKGE}
Then we check the convergence patterns of the finite difference scheme: EFD (\ref{EFD}) to the RLogKGE (\ref{RLogKGE})  for various mesh size $h$, time step $\tau$ under any fixed
parameter $0<\varepsilon\ll1$ for Example~1 and Example~2.

Firstly, we perform test on the temporal errors with the EFD (\ref{EFD}) in the $l^{2}$-norm, $l^{\infty}$-norm, $H^{1}$-norm at $T=1$, depicted in Figure \ref{fig:efdcase2error}. Due to the stability condition of the EFD (\ref{EFD}), we set $0< \tau <\min\{\frac{1}{2}, \frac{h}{2}\}$, varying the mesh size and time step simultaneously as $\tau_{j}=0.01\times2^{-j}, h_{j}=2^{-j}$ for $j=1,\ldots,7$.

Secondly, for spatial accuracy of the EFD (\ref{EFD}) at $T=1$, we set time step $\tau=\tau_{e}=0.01\times{2^{-9}}$, such that the errors from the time discretization are ignored and solve the RLogKGE (\ref{RLogKGE}) with the FDTD schemes versus mesh size $h$. The results are displayed in Figure \ref{fig:regefdcase2spaceerr}. Figure \ref{fig:regefdcase2spaceerr} depict $\|e^{\varepsilon}\|_{l^{2}}, \|e^{\varepsilon}\|_{l^{\infty}},\|e^{\varepsilon}\|_{H^{1}}$ with different $h$ of the scheme EFD (\ref{EFD}) for Example~1 and Example~2.

From Figure \ref{fig:regefdcase2spaceerr}, we can make the observations:
the scheme EFD (\ref{EFD}) are uniformly second order accurate for the RLogKGE (\ref{RLogKGE}) for any $0<\varepsilon\ll 1$ which demonstrate the Theorem \ref{EFDerr}.
\begin{center}
\begin{figure}[htbp]
\centering
\subfigure{\includegraphics[width=.33\textwidth, height=0.30\textwidth]{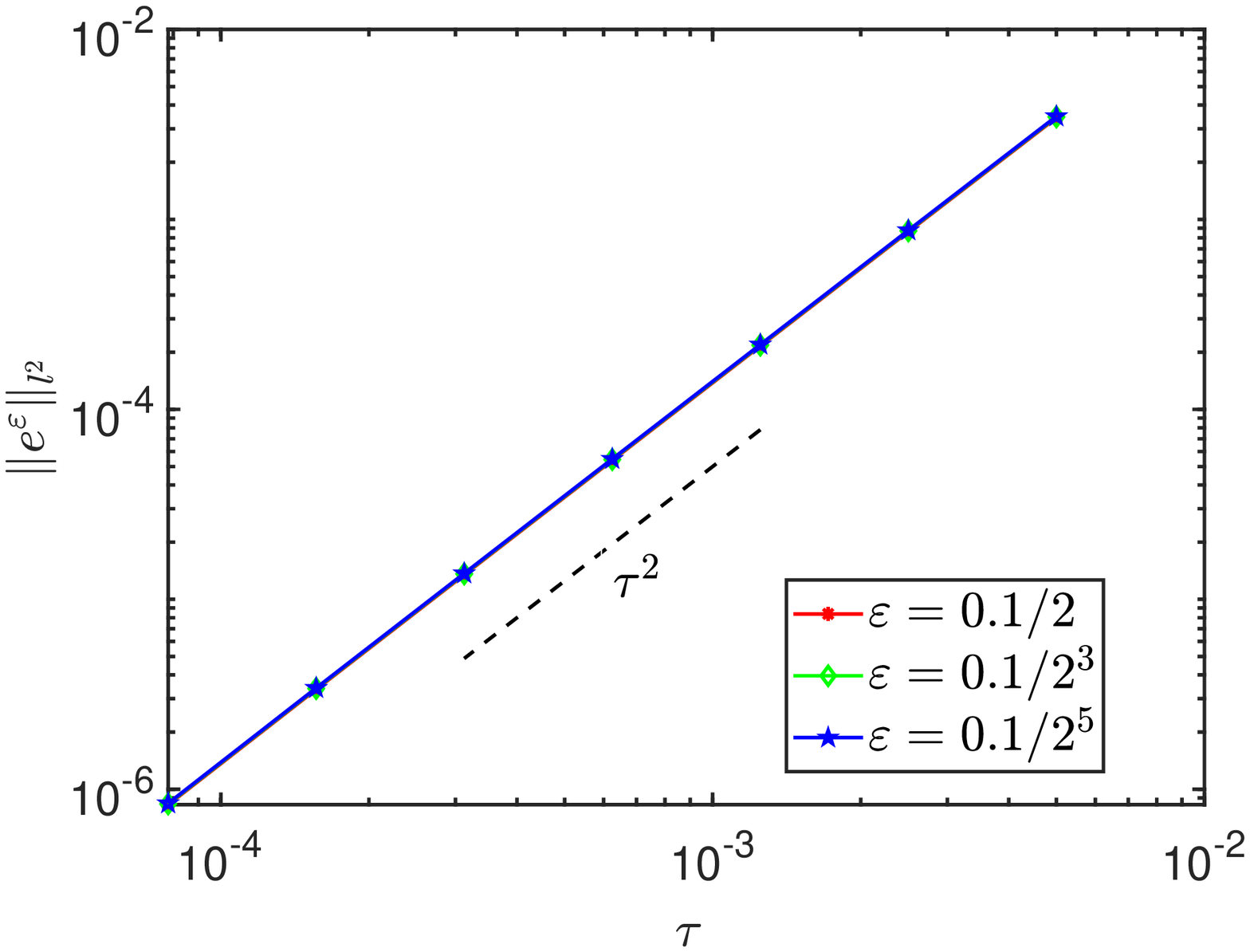}}~~~
\subfigure{\includegraphics[width=.33\textwidth, height=0.30\textwidth]{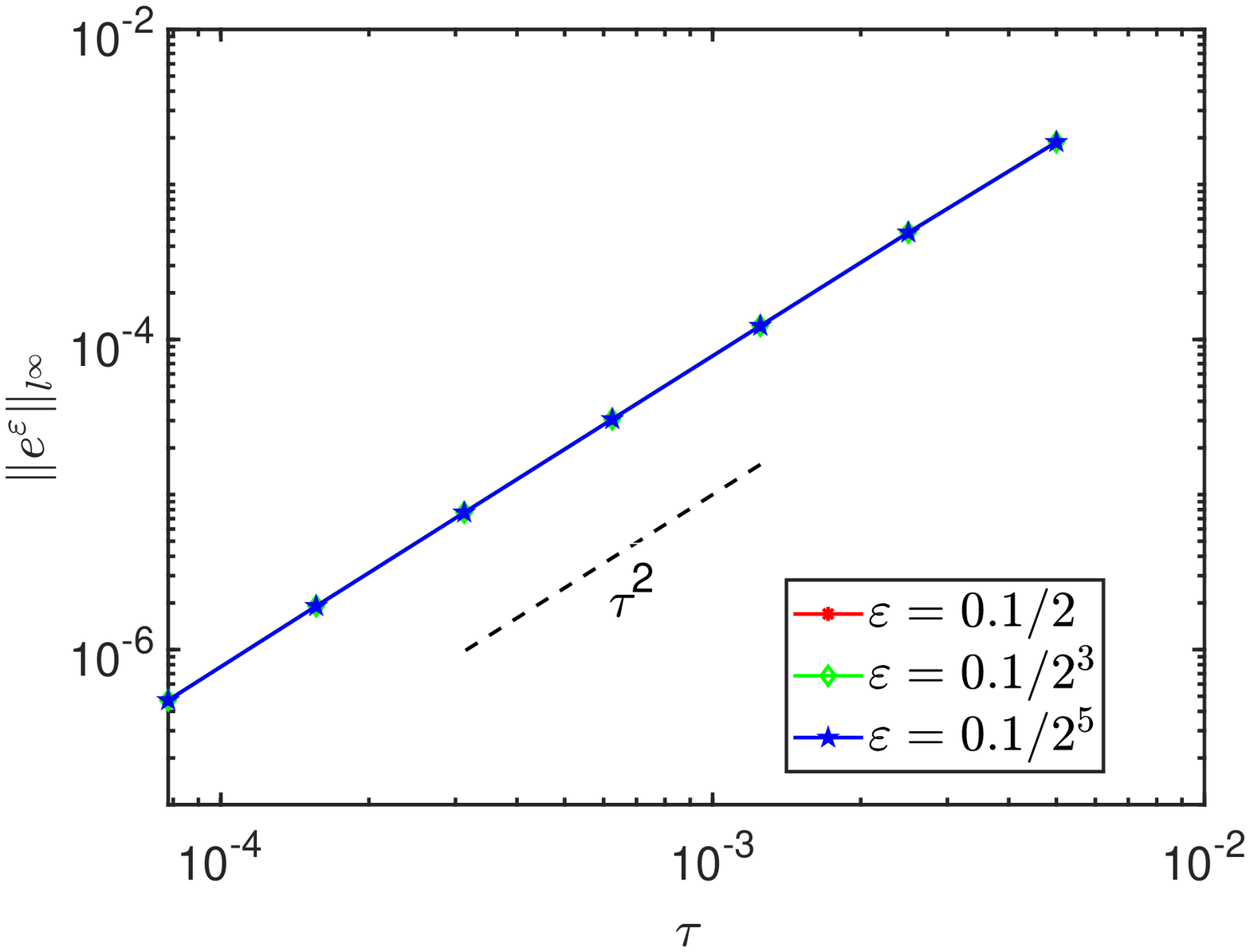}}~~~
\subfigure{\includegraphics[width=.33\textwidth, height=0.30\textwidth]{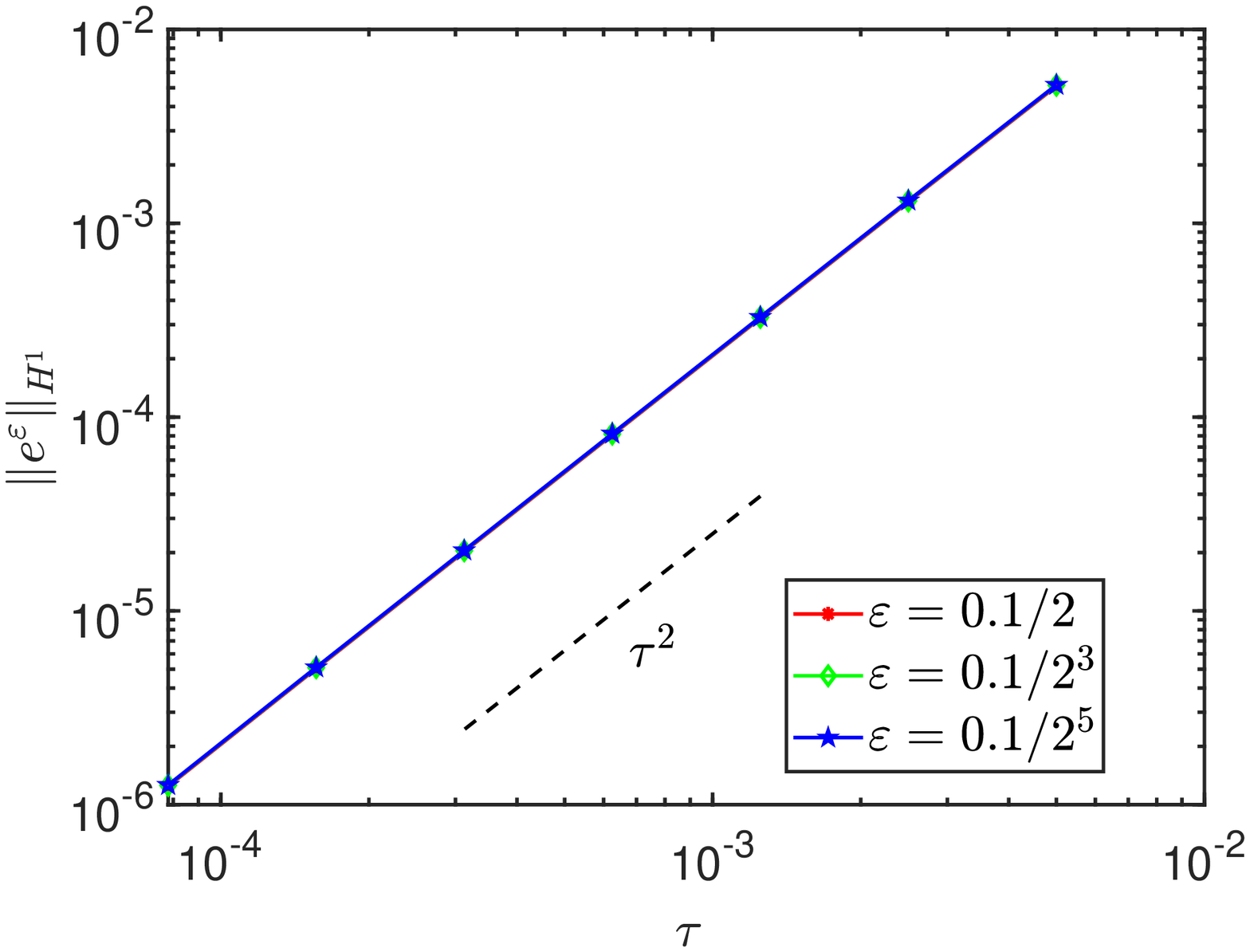}}
\caption{The temporal errors $e^{\varepsilon}(1)$ in three different norms for Example~1.}
\label{fig:regefderr}
\end{figure}
\end{center}
\begin{center}
\begin{figure}[htbp]
\centering
\subfigure{\includegraphics[width=.33\textwidth, height=0.30\textwidth]{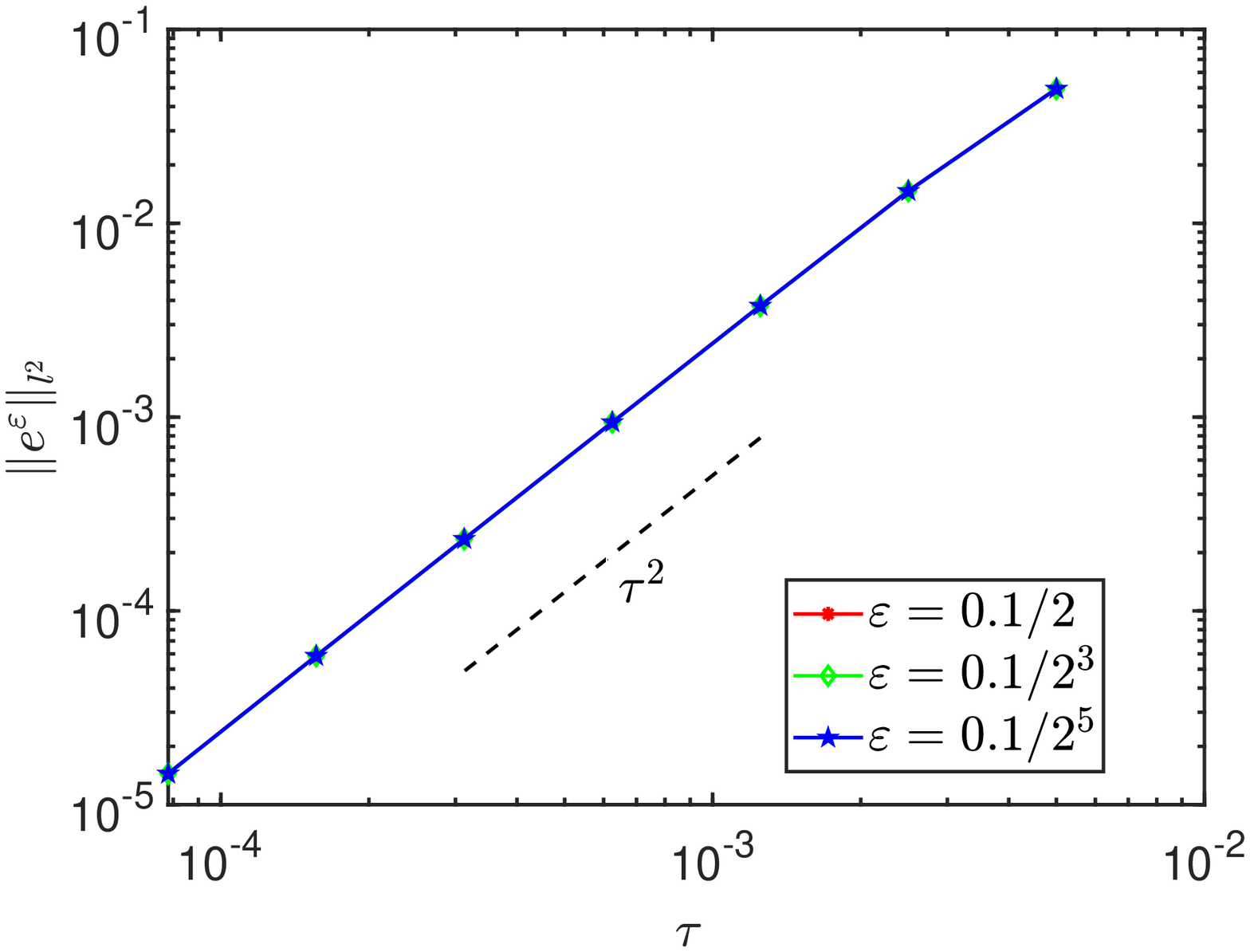}}~~
\subfigure{\includegraphics[width=.33\textwidth, height=0.30\textwidth]{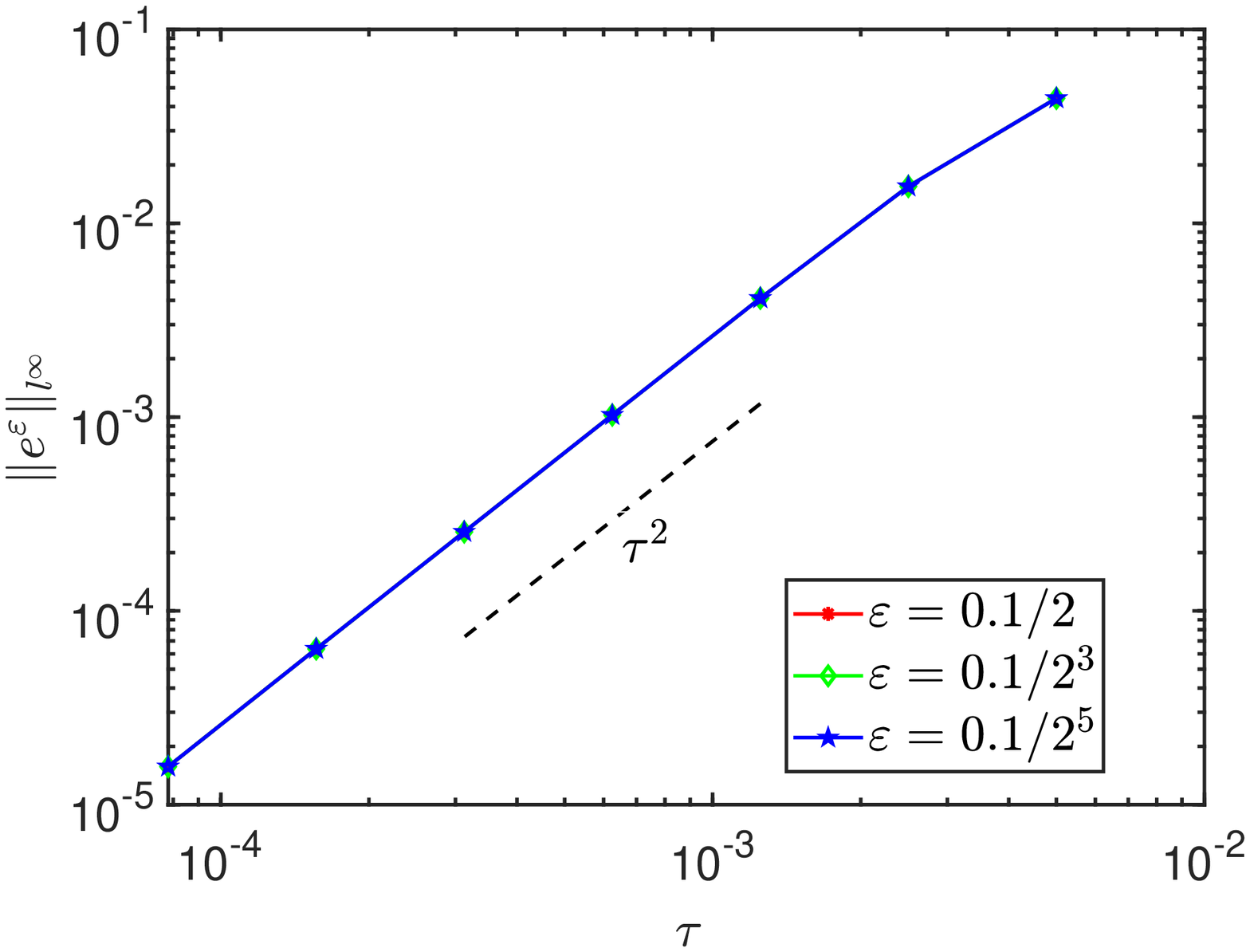}}~~
\subfigure{\includegraphics[width=.33\textwidth, height=0.30\textwidth]{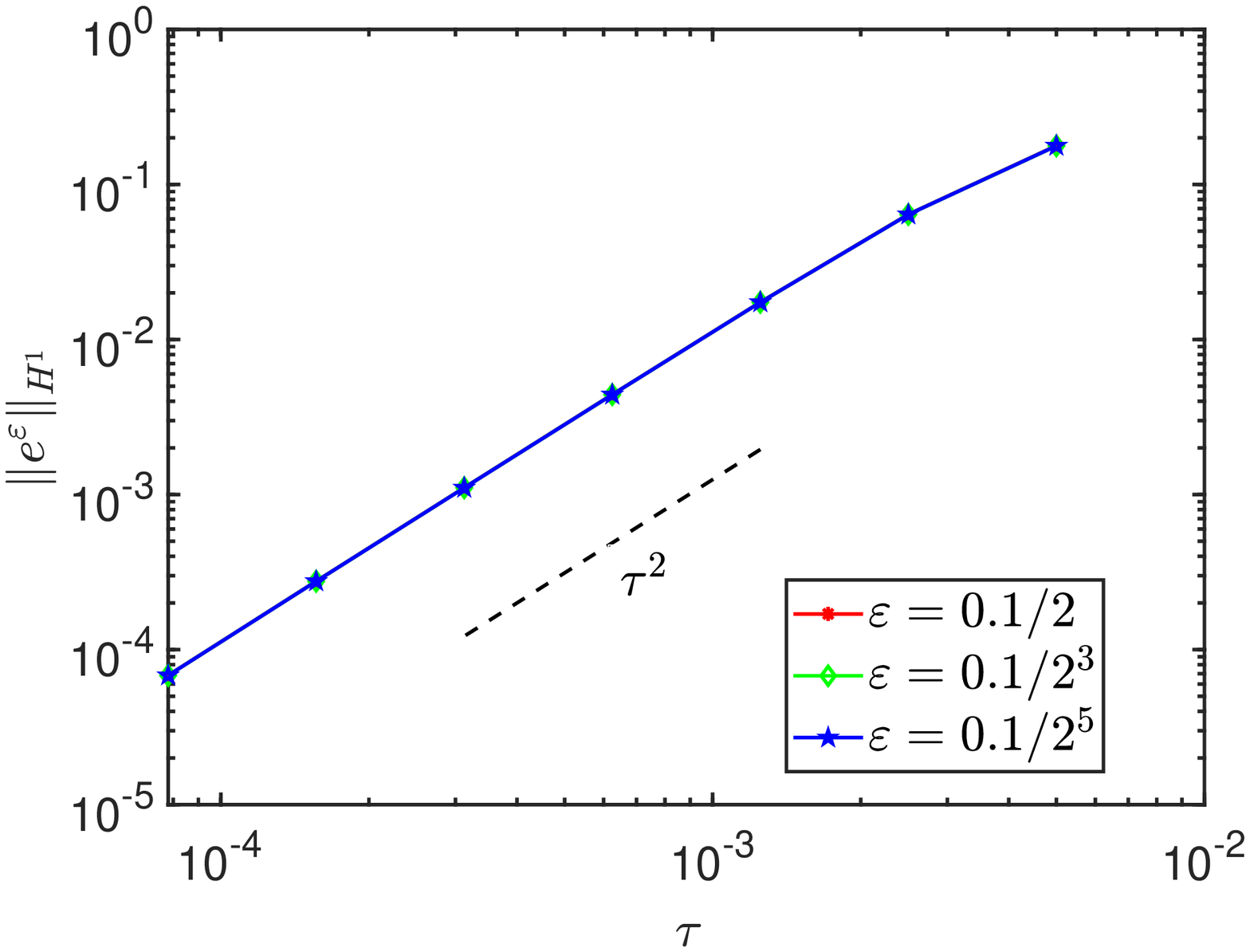}}
\caption{The temporal errors $e^{\varepsilon}(1)$ in three different norms for Example~2.}
\label{fig:efdcase2error}
\end{figure}
\end{center}
\begin{center}
\begin{figure}[htbp]
\centering
\subfigure{\includegraphics[width=.33\textwidth, height=0.30\textwidth]{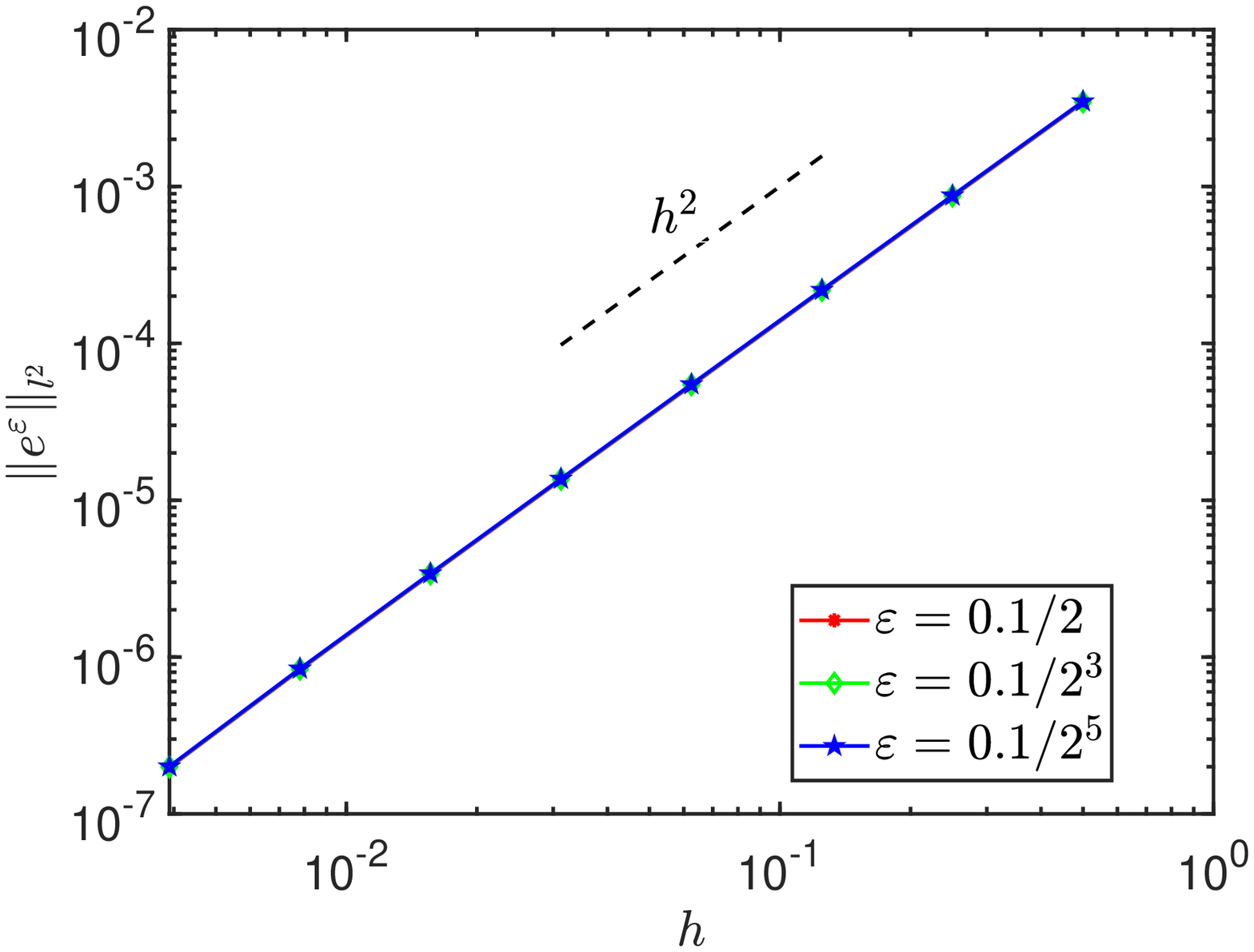}}~~~
\subfigure{\includegraphics[width=.33\textwidth, height=0.30\textwidth]{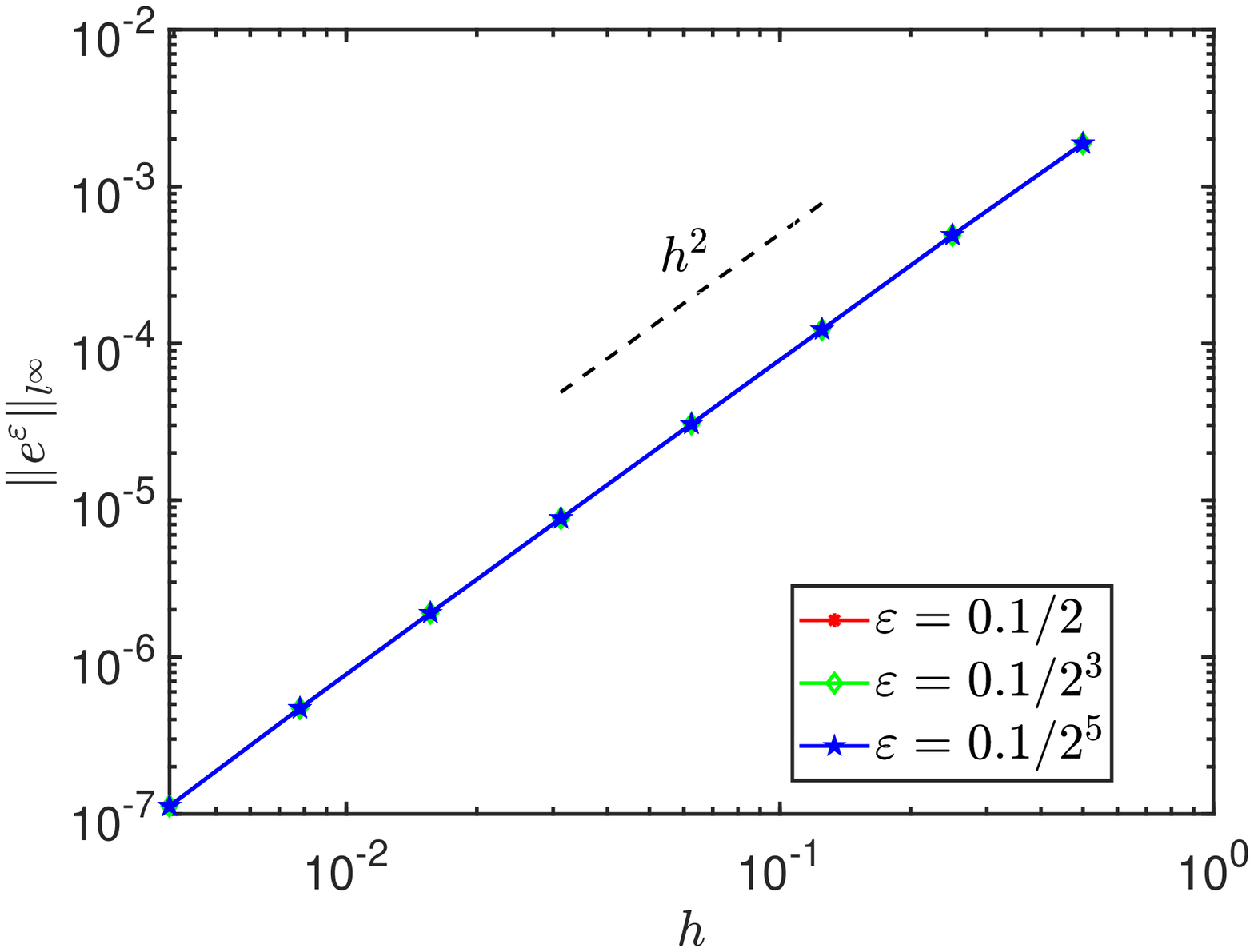}}~~~
\subfigure{\includegraphics[width=.33\textwidth, height=0.30\textwidth]{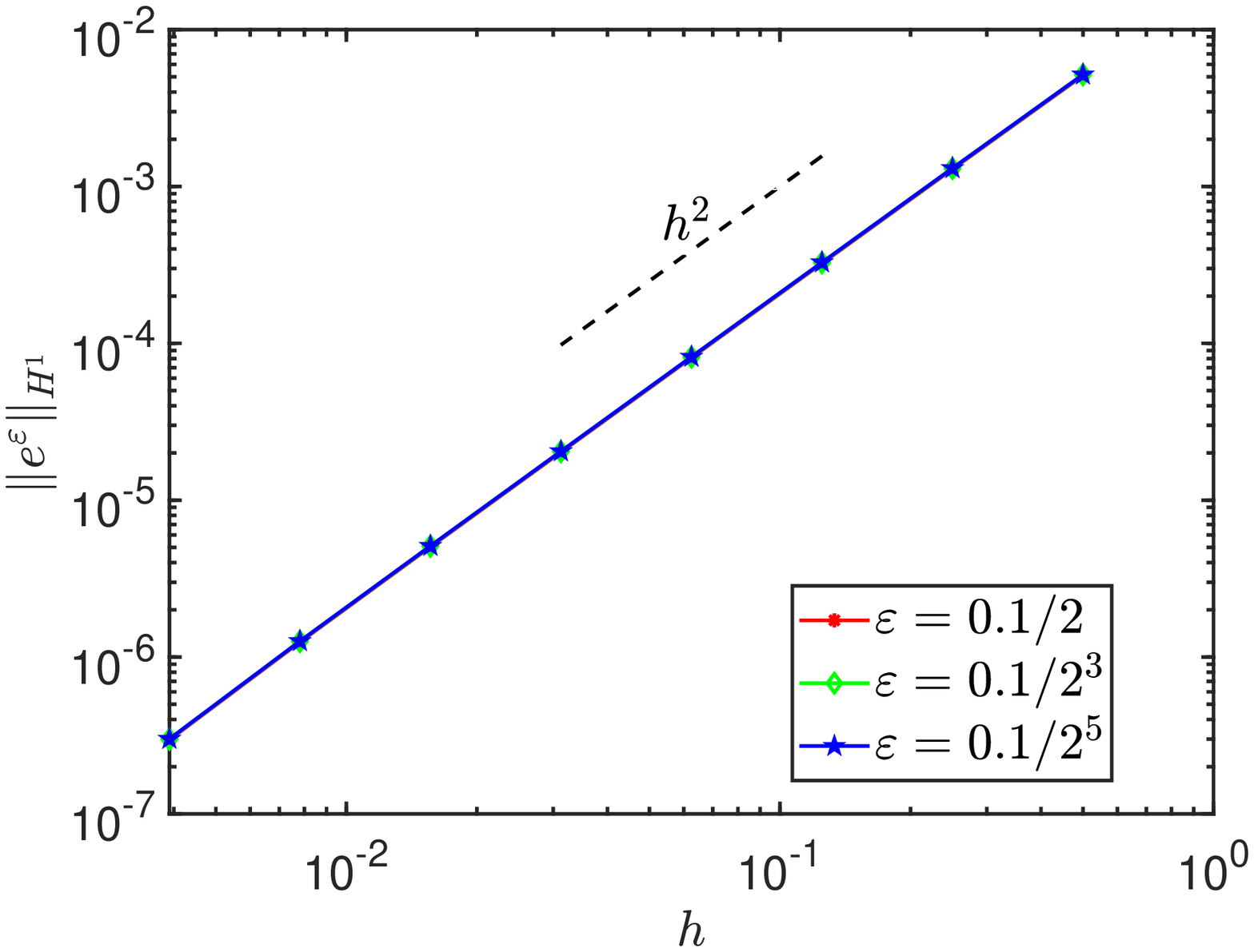}}
\caption{The spatial errors $e^{\varepsilon}(1)$ in three different norms for Example~1.}
\label{fig:regefdspaceerr}
\end{figure}
\end{center}

\begin{center}
\begin{figure}[htbp]
\centering
\subfigure{\includegraphics[width=.33\textwidth, height=0.30\textwidth]{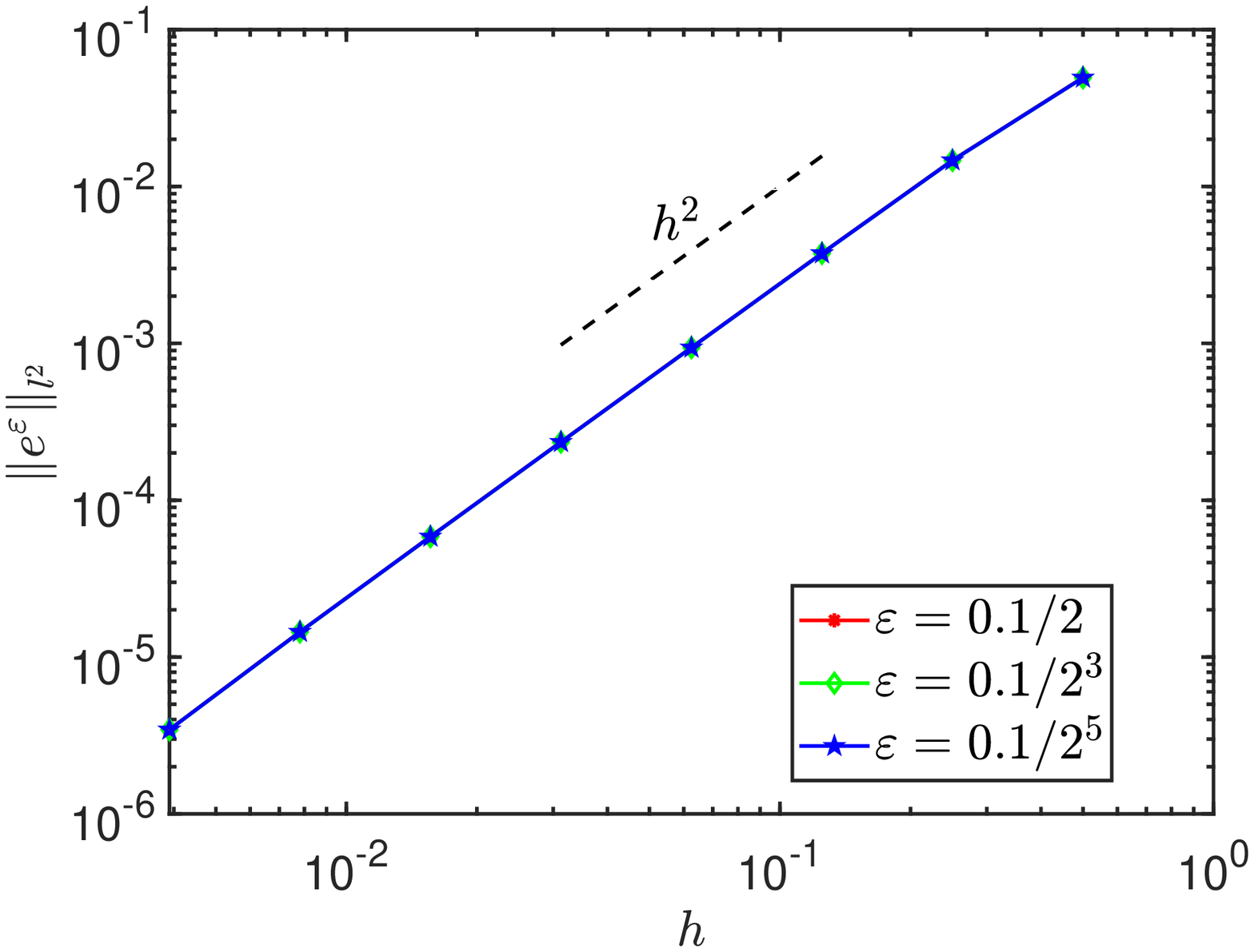}}~~~
\subfigure{\includegraphics[width=.33\textwidth, height=0.30\textwidth]{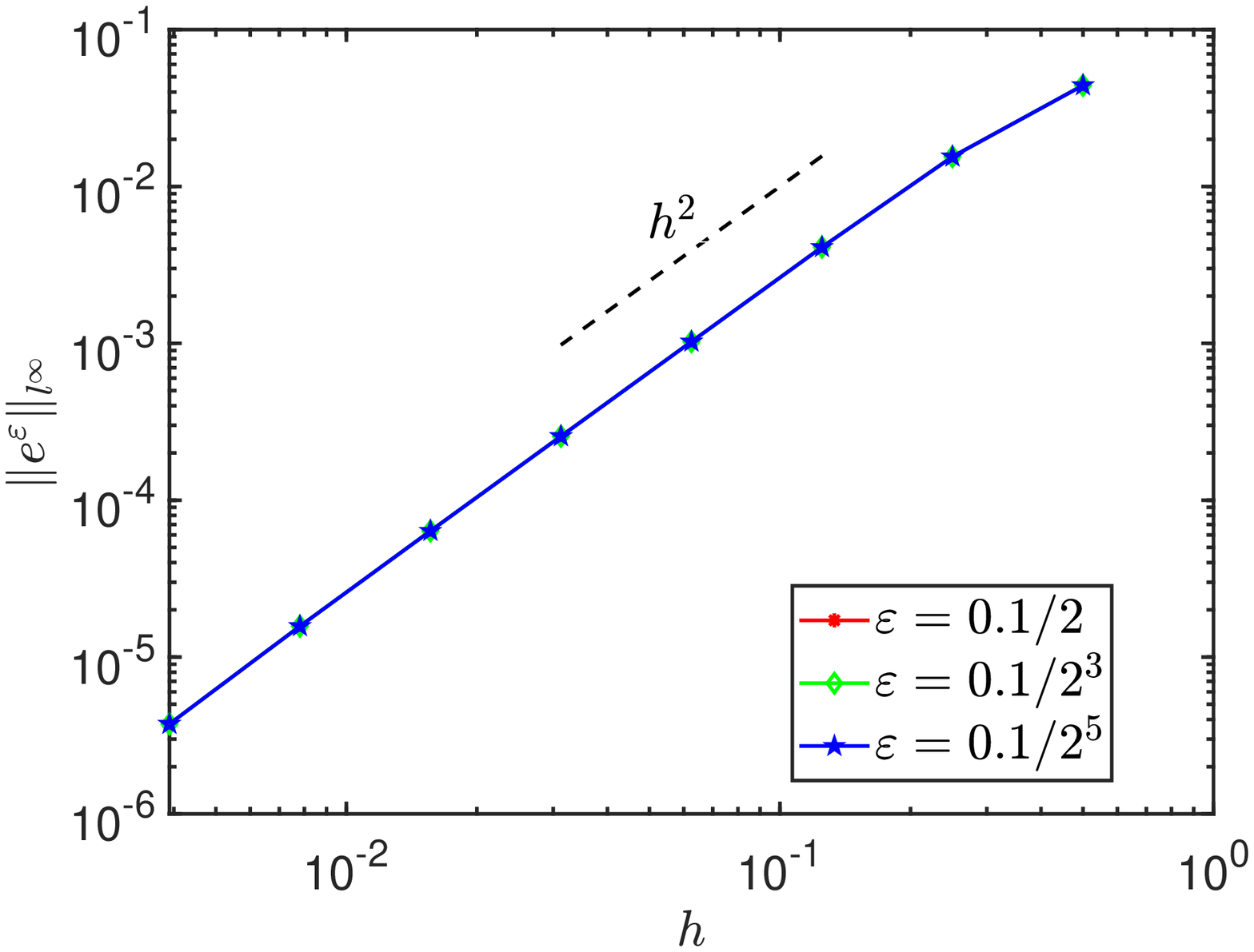}}~~~
\subfigure{\includegraphics[width=.33\textwidth, height=0.30\textwidth]{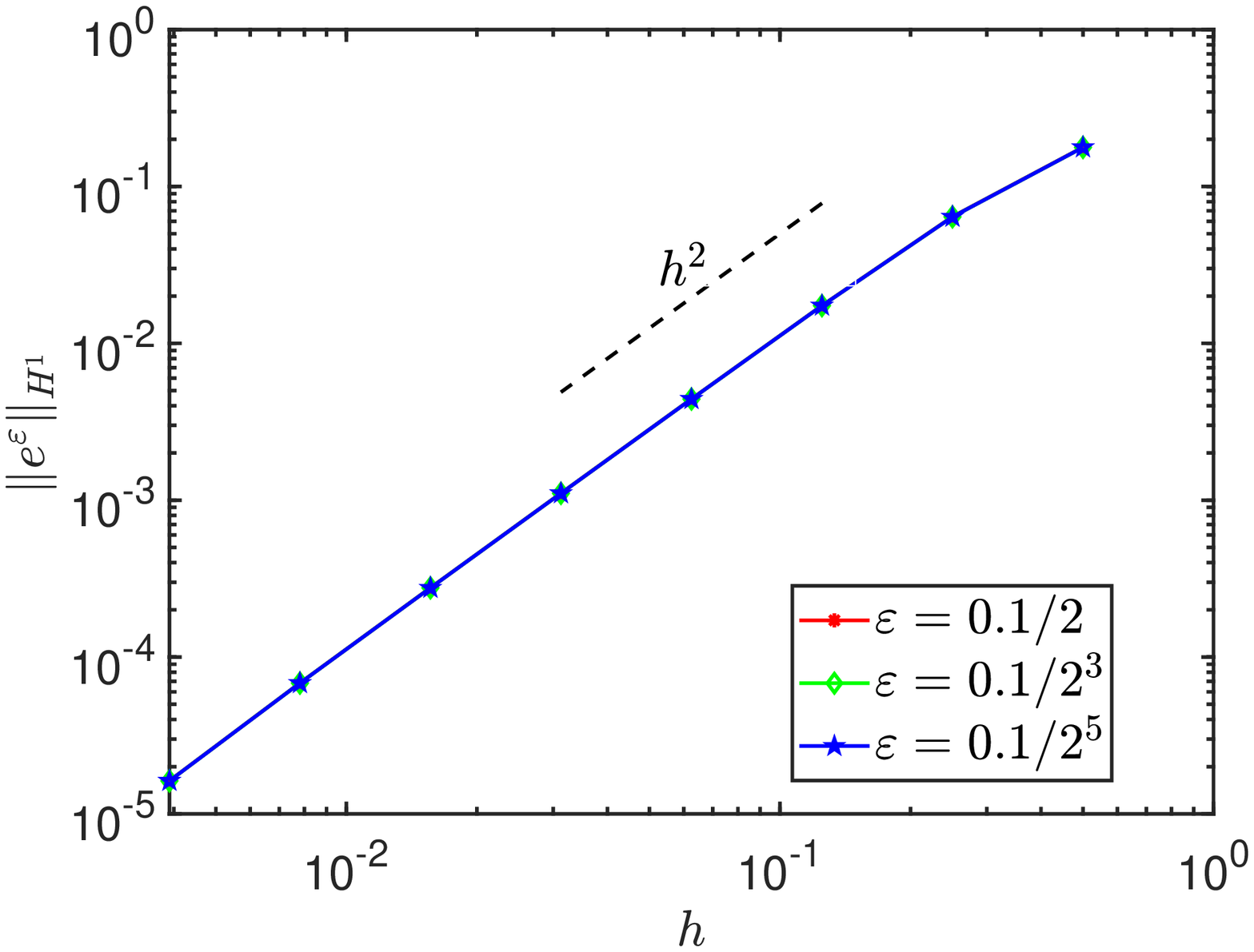}}
\caption{The spatial errors $e^{\varepsilon}(1)$ in three different norms for Example~2.}
\label{fig:regefdcase2spaceerr}
\end{figure}
\end{center}
\subsection{Convergence of FDTD to the LogKGE}
We check the convergence rates of the finite difference schemes: SIFD (\ref{SIFD2}) and EFD (\ref{EFD}) to the LogKGE (\ref{LogKGE}) for Example~1. Tables \ref{SIFD2He} and \ref{EFDHe} display $l^{2}$-norm, $l^{\infty}$-norm, $H^{1}$-norm of $\tilde{e}^{\varepsilon}(1)$, respectively, for various mesh size $h$, time step $\tau$ and parameter $\varepsilon$.

\begin{table}
\centering
\caption{The convergence of the SIFD (\ref{SIFD2}) scheme to the LogKGE (\ref{LogKGE}) with different $\tau,h,\varepsilon$}
\begin{tabular}{cccccccccccc}
\toprule  
$\|\tilde{e}^{\varepsilon} (1)\|_{l^{\infty}} $ &$  h=0.1 $ & $ h/2   $& $h/2^{2}   $ & $h/2^{3}    $& $h/2^{4}   $ &$h/2^{5}   $& \\
$ $ &$\tau=0.1$ & $\tau/2 $& $\tau/2^{2}$ & $\tau/2^{3} $& $\tau/2^{4}$ &$\tau/2^{5}$&\\
\midrule  
$\varepsilon=10^{-3} $&$\mathbf{4.03E}$-$\mathbf{03}$ &1.03E-03& 7.34E-04 &7.66E-04 & 7.74E-04 &7.76E-04\\
 rate   &--& $1.97$ & 0.48 & -0.06 &-0.02&0.00\\\hline
$\varepsilon/4$ &4.03E-03 &$\mathbf{1.03E}$-$\mathbf{03}$& 2.59E-04 &2.24E-04 & 2.28E-04 &2.29E-04\\
 rate   &--& $\mathbf{1.97}$ & $1.99$ & $0.21$ &-0.03&-0.01\\\hline
$\varepsilon/4^{2}$&4.03E-03 &1.03E-03& $\mathbf{2.59E}$-$\mathbf{04}$ &6.52E-05& 6.70E-05 &6.75E-05\\
 rate   &--& $1.97$ & $\mathbf{1.99}$ & $1.99$ &-0.04&-0.01\\\hline
$\varepsilon/4^{3}$ &4.03E-03 &1.03E-03& 2.59E-04 &$\mathbf{6.50E}$-$\mathbf{05}$ & 1.96E-05 &1.98E-05\\
 rate   &--& $1.97$ & 1.99 & $\mathbf{1.99}$ &$1.73$&-0.01\\\hline
$\varepsilon/4^{4}$ &4.03E-03 &1.03E-03& 2.59E-04  &6.50E-05 & $\mathbf{1.63E}$-$\mathbf{05}$ &5.81E-06\\
 rate   &--& $1.97$ & 1.99 & $1.99$ &$\mathbf{2.00}$&1.49\\
\toprule  
$\|\tilde{e}^{\varepsilon}(1) \|_{l^{2}} $ &$  h=0.1 $ & $ h/2   $& $h/2^{2}   $ & $h/2^{3}    $& $h/2^{4}   $ &$h/2^{5}   $& \\
$ $ &$\tau=0.1$ & $\tau/2 $& $\tau/2^{2}$ & $\tau/2^{3} $& $\tau/2^{4}$ &$\tau/2^{5}$&\\
\midrule  
$\varepsilon=10^{-3} $&$\mathbf{7.72E}$-$\mathbf{03}$ &2.23E-03& 1.30E-03 &1.24E-03 & 1.25E-03 &1.25E-03\\
 rate   &--& $1.80$ & 0.78 & $0.06$ &-0.01&0.00\\\hline
$\varepsilon/4$ &7.73E-03 &$\mathbf{1.99E}$-$\mathbf{03}$& 5.93E-04 &3.69E-04 & 3.56E-04 &3.56E-04\\
 rate   &--& $\mathbf{1.96}$ & $1.75$ & $0.68$ &0.05&0.00\\\hline
$\varepsilon/4^{2}$&7.74E-03 &1.98E-03& $\mathbf{5.07E}$-$\mathbf{04}$ &1.59E-04 & 1.05E-04 &1.02E-04\\
 rate   &--& $1.97$ & $\mathbf{1.97}$ & $1.68$ &0.59&0.05\\\hline
$\varepsilon/4^{3}$ &7.74E-03 &1.98E-03& 5.02E-04 &$\mathbf{1.29E}$-$\mathbf{04}$ & 4.25E-05 &3.01E-05\\
 rate   &--& $1.96$ & 1.98 & $\mathbf{1.96}$ &$1.60$&0.50\\\hline
$\varepsilon/4^{4}$ &7.74E-03 &1.98E-03& 5.01E-04 &1.26E-04 & $\mathbf{3.26E}$-$\mathbf{05}$ &1.14E-05\\
 rate   &--& $1.96$ & 1.98 & $1.99$ &$\mathbf{1.95}$&1.51\\
\toprule  
$\|\tilde{e}^{\varepsilon}(1) \|_{H^{1}}$ &$  h=0.1 $ & $ h/2   $& $h/2^{2}   $ & $h/2^{3}    $& $h/2^{4}   $ &$h/2^{5}   $& \\
$ $ &$\tau=0.1$ & $\tau/2 $& $\tau/2^{2}$ & $\tau/2^{3} $& $\tau/2^{4}$ &$\tau/2^{5}$&\\
\midrule  
$\varepsilon=10^{-3} $&$\mathbf{1.08E}$-$\mathbf{02}$ &3.07E-03& 1.68E-03 &1.59E-03 & 1.59E-03 &1.59E-03\\
 rate   &--& $1.81$ & 0.87 & $0.09$ &0.00&0.00\\\hline
$\varepsilon/4$ &1.08E-02 &$\mathbf{2.79E}$-$\mathbf{03}$& 8.26E-04 &4.91E-04 & 4.65E-04 &4.64E-04\\
 rate   &--& $\mathbf{1.95}$ & $1.76$ & $0.75$ &0.08&0.00\\\hline
$\varepsilon/4^{2}$&1.08E-02 &2.77E-03& $\mathbf{7.10E}$-$\mathbf{04}$ &2.21E-04 & 1.42E-04 &1.36E-04\\
 rate   &--& $1.97$ & $\mathbf{1.96}$ & $1.68$ &0.64&0.07\\\hline
$\varepsilon/4^{3}$ &1.08E-02 &2.76E-03& 6.99E-04 &$\mathbf{1.80E}$-$\mathbf{04}$ & 5.92E-05 &4.10E-05\\
 rate   &--& $1.97$ & 1.98 & $\mathbf{1.96}$ &$1.60$&0.53\\\hline
$\varepsilon/4^{4}$ &1.08E-02 &2.76E-03& 6.98E-04 &1.76E-04 & $\mathbf{4.54E}$-$\mathbf{05}$ &1.60E-05\\
 rate   &--& $1.97$ & 1.99 & $1.99$ &$\mathbf{1.95}$&1.51\\
\bottomrule 
\end{tabular}
\label{SIFD2He}
\end{table}
\begin{table}
\centering
\caption{The convergence of the EFD (\ref{EFD}) scheme to the LogKGE (\ref{LogKGE}) with different $\tau,h,\varepsilon$}
\begin{tabular}{cccccccccccc}
\toprule  
$\|\tilde{e}^{\varepsilon} (1)\|_{l^{\infty}} $ &$  h=0.1 $ & $ h/2   $& $h/2^{2}   $ & $h/2^{3}    $& $h/2^{4}   $ &$h/2^{5}   $& \\
$ $ &$\tau=0.1$ & $\tau/2 $& $\tau/2^{2}$ & $\tau/2^{3} $& $\tau/2^{4}$ &$\tau/2^{5}$&\\
\midrule  
$\varepsilon=10^{-3} $&$\mathbf{1.63E}$-$\mathbf{03}$ &6.76E-04& 7.43E-04 &7.68E-04 & 7.75E-04 &7.76E-03\\
 rate   &--& $1.27$ & -0.14 & -0.05 &-0.01&0.00\\\hline
$\varepsilon/4$ &1.70E-03 &$\mathbf{4.31E}$-$\mathbf{04}$& 2.14E-04 &2.25E-04 & 2.28E-04 &2.29E-04\\
 rate   &--& $\mathbf{1.98}$ & $1.01$ & -0.08 &-0.02&0.00\\\hline
$\varepsilon/4^{2}$&1.71E-03 &4.37E-04& $\mathbf{1.10E}$-$\mathbf{04}$ &6.58E-05 & 6.72E-05 &6.76E-05\\
 rate   &--& $1.97$ & $\mathbf{1.99}$ & $0.73$ &-0.03&-0.01\\\hline
$\varepsilon/4^{3}$ &1.71E-03 &4.37E-03& 1.10E-04 &$\mathbf{2.76E}$-$\mathbf{04}$ & 1.97E-05 &1.99E-05\\
 rate   &--& $1.96$ & 1.99 & $\mathbf{2.00}$ &$0.48$&-0.01\\\hline
$\varepsilon/4^{4}$ &1.71E-03 &4.37E-04& 1.10E-04 &2.76E-05 & $\mathbf{6.90E}$-$\mathbf{06}$ &5.82E-06\\
 rate   &--& $1.96$ & 1.99 & $2.00$ &$\mathbf{2.00}$&0.25\\
 \toprule  
 $\|\tilde{e}^{\varepsilon} (1) \|_{l^{2}}$ &$  h=0.1 $ & $ h/2   $& $h/2^{2}   $ & $h/2^{3}    $& $h/2^{4}   $ &$h/2^{5}   $& \\
$ $ &$\tau=0.1$ & $\tau/2 $& $\tau/2^{2}$ & $\tau/2^{3} $& $\tau/2^{4}$ &$\tau/2^{5}$&\\
\midrule  
$\varepsilon=10^{-3} $&$\mathbf{3.73E}$-$\mathbf{03}$ &1.43E-03& 1.23E-03 &1.24E-03 & 1.25E-03 &1.25E-03\\
 rate   &--& $1.39$ & 0.21 & -0.01 &-0.01&0.00\\\hline
$\varepsilon/4$ &3.71E-03 &$\mathbf{9.72E}$-$\mathbf{04}$& 4.06E-04 &3.55E-04 & 3.56E-04 &3.56E-04\\
 rate   &--& $\mathbf{1.93}$ & $1.26$ & $0.19$ &0.00&0.00\\\hline
$\varepsilon/4^{2}$&3.72E-03 &9.43E-04& $\mathbf{2.52E}$-$\mathbf{04}$ &1.15E-04 & 1.02E-04 &1.02E-04\\
 rate   &--& $1.98$ & $\mathbf{1.90}$ & $1.13$ &0.17&0.00\\\hline
$\varepsilon/4^{3}$ &3.72E-03 &9.42E-04& 2.38E-04 &$\mathbf{6.53E}$-$\mathbf{05}$ & 3.23E-05 &2.93E-05\\
 rate   &--& $1.98$ & 1.99 & $\mathbf{1.86}$ &$1.01$&0.14\\\hline
$\varepsilon/4^{4}$ &3.72E-03 &9.43E-04& 2.37E-04 &5.98E-05 & $\mathbf{1.69E}$-$\mathbf{05}$ &9.10E-06\\
 rate   &--& $1.98$ & 1.99 & $1.99$ &$\mathbf{1.82}$&0.89\\
 \toprule  
$\|\tilde{e}^{\varepsilon} (1)\|_{H^{1}} $ &$  h=0.1 $ & $ h/2   $& $h/2^{2}   $ & $h/2^{3}    $& $h/2^{4}   $ &$h/2^{5}   $& \\
$ $ &$\tau=0.1$ & $\tau/2 $& $\tau/2^{2}$ & $\tau/2^{3} $& $\tau/2^{4}$ &$\tau/2^{5}$&\\
\midrule  
$\varepsilon=10^{-3} $&5.05E-03 &1.92E-03& 1.59E-03 &1.59E-03 & 1.59E-03 &1.59E-03\\
 rate   &--& $1.39$ & 0.27 & $0.01$ &0.00&0.00\\\hline
$\varepsilon/4$ &$\mathbf{4.93E}$-$\mathbf{03 } $&1.31E-03& 5.51E-04&4.68E-04 &4.64E-04 & 4.64E-04 \\
 rate   &--& $1.91$ & $1.25$ & $0.23$ &0.01&0.00\\\hline
$\varepsilon/4^{2}$&4.92E-03 &$\mathbf{1.25E}$-$\mathbf{03}$& 3.39E-04 &1.56E-04 & 1.37E-04 &1.35E-04\\
 rate   &--& $\mathbf{1.98}$ & $1.88$ & $1.12$ &0.19&0.01\\\hline
$\varepsilon/4^{3}$ &4.92E-03 &1.24E-03& $\mathbf{3.14E}$-$\mathbf{04}$ &8.76E-05 & 4.12E-05 &3.98E-05\\
 rate   &--& $1.99$ & $\mathbf{1.98}$ & 1.84 &$0.99$&0.15\\\hline
$\varepsilon/4^{4}$ &4.92E-03 &1.24E-03& 3.11E-04 &$\mathbf{7.87E}$-$\mathbf{05}$ & 2.27E-05 &1.25E-05\\
 rate   &--& $1.99$ & 1.99 & $\mathbf{1.98}$ &$1.78$&0.86\\\hline
 $\varepsilon/4^{5}$ &4.92E-03 &1.24E-03& 3.11E-04 &7.79E-05 & $\mathbf{1.96E}$-$\mathbf{05}$ &5.92E-06\\
 rate   &--& $1.99$ & 2.00 & $2.00$ &$\mathbf{1.98}$&1.74\\
\bottomrule 
\end{tabular}
\label{EFDHe}
\end{table}
\subsection{The evolution of the solution}
Figure \ref{fig:efdcase2ut1} represents the numerical solutions of the EFD (\ref{EFD}) at three different time $T=1,5,9$ for Example~2. We take the step size as $\tau=0.01\times2^{-7}$, and the mesh size as $h=2^{-7}$ at the large domain $[-16,16]$. From Figure \ref{fig:efdcase2ut1}, we can see that the numerical solutions of those two schemes are very close different $\varepsilon$ at fixed times. Besides the number of wave crests increase over time. We can conclude the two discritization schemes are stable under the stability conditions.
\begin{center}
\begin{figure}[htbp]
\centering
\subfigure[$T=1$]{\includegraphics[width=.33\textwidth, height=0.30\textwidth]{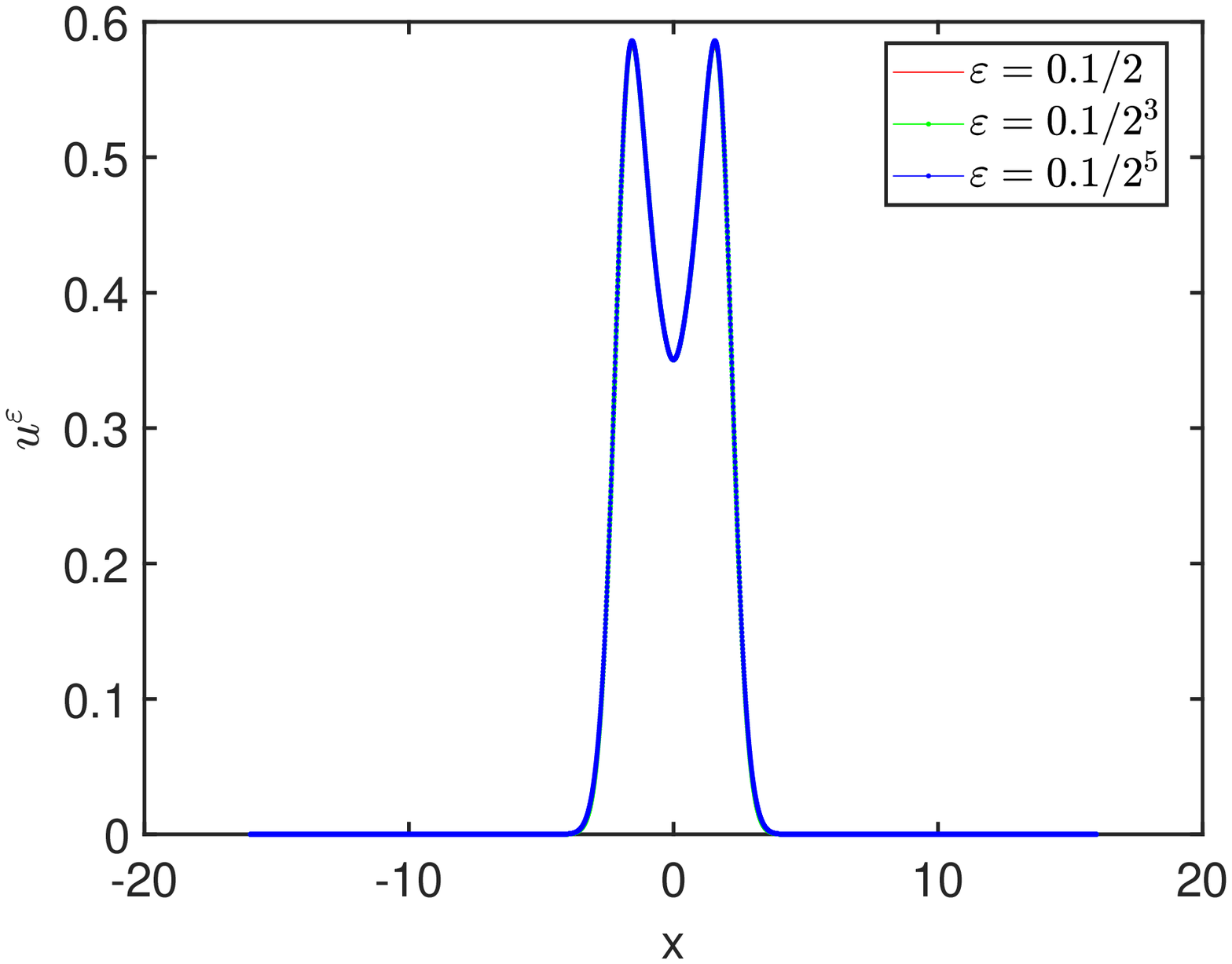}}~~
\subfigure[$T=5$]{\includegraphics[width=.33\textwidth, height=0.30\textwidth]{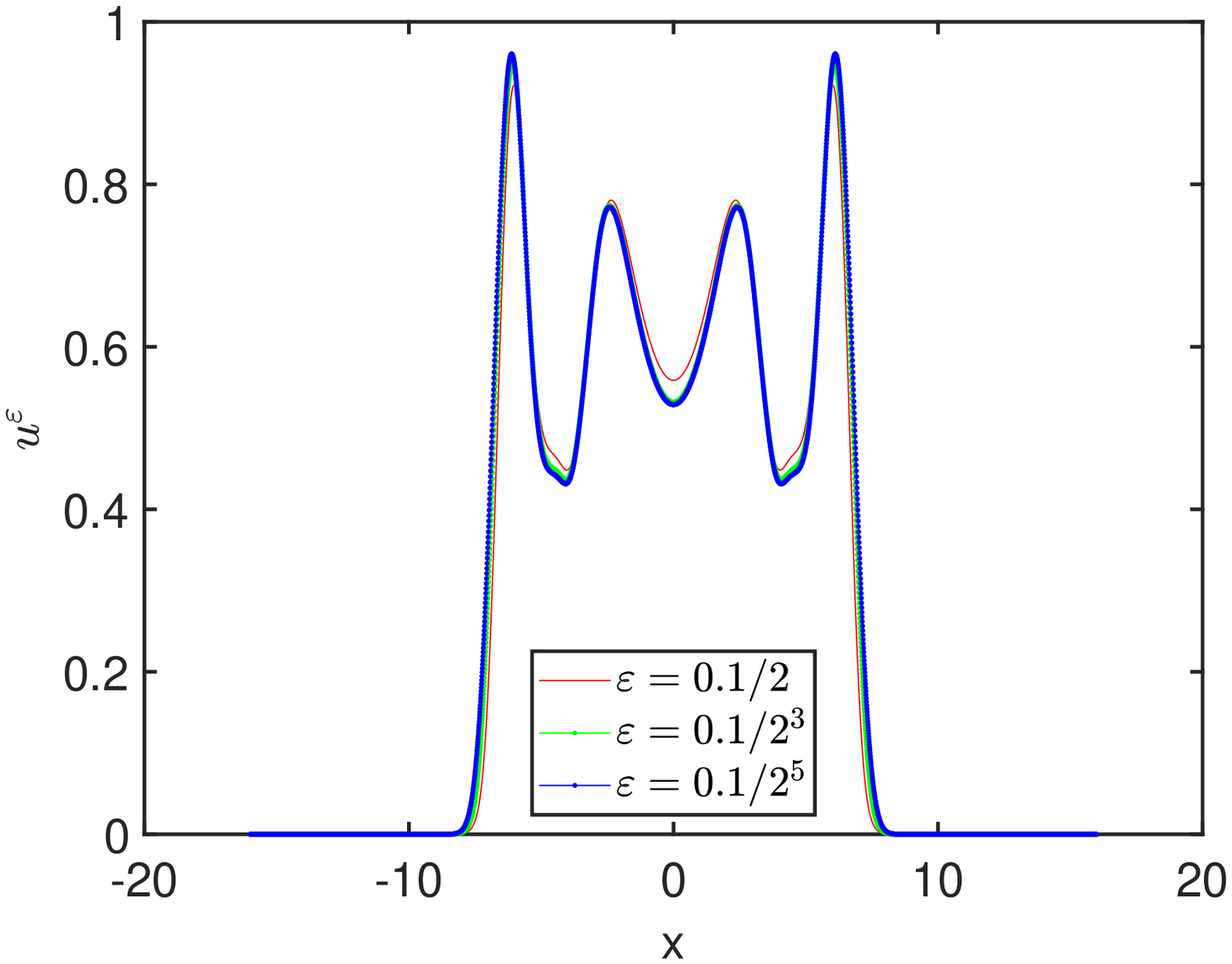}}~~
\subfigure[$T=9$]{\includegraphics[width=.33\textwidth, height=0.30\textwidth]{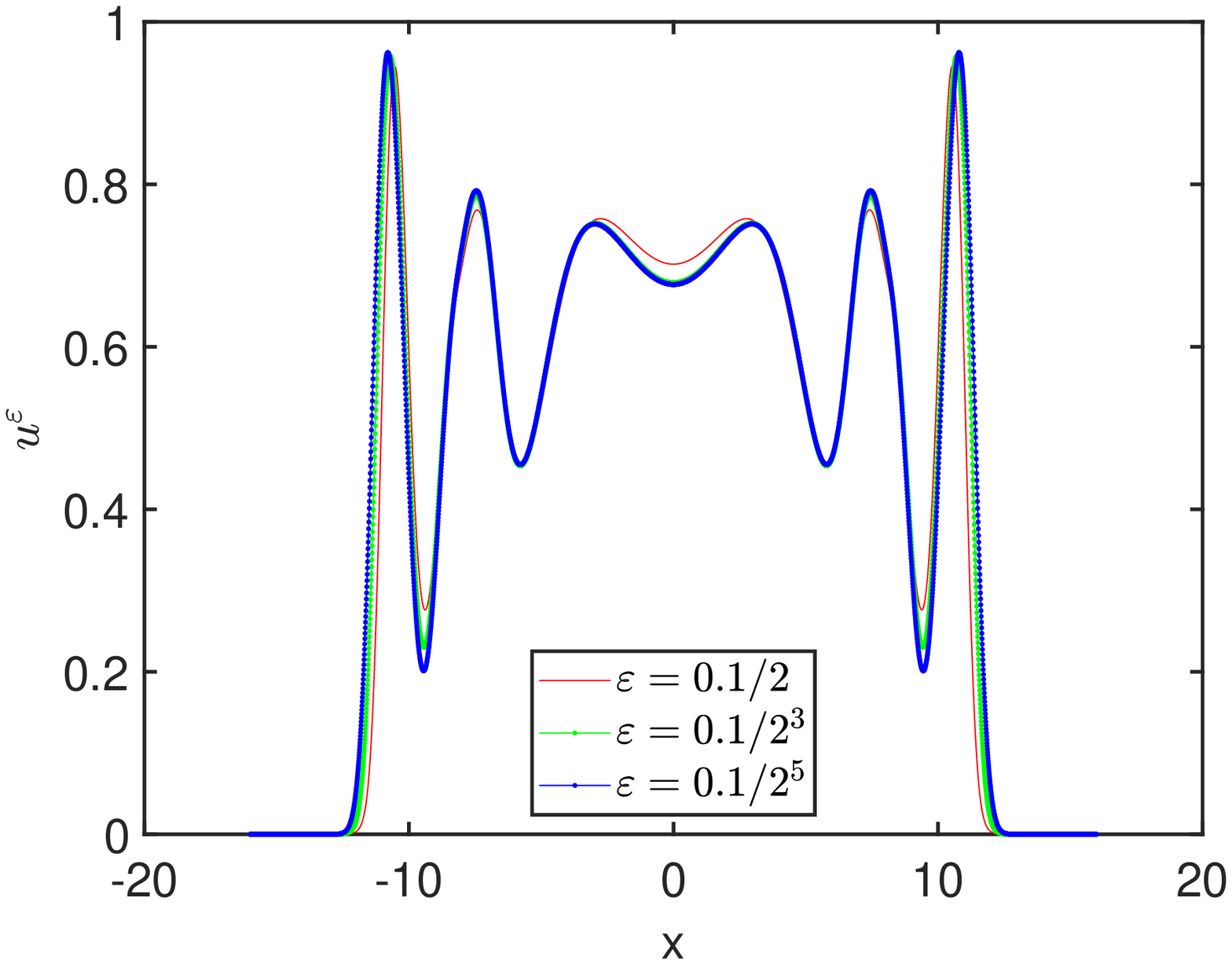}}
\caption{The numerical solution $u^{\varepsilon}$ in three different time for the scheme EFD $(\ref{EFD})$ .}
\label{fig:efdcase2ut1}
\end{figure}
\end{center}

\section{Conclusions}\label{conclusion}
To avoid the singularity of the LogKGE (\ref{LogKGE}) at the origin, we proposed the RLogKGE (\ref{RLogKGE}) with a small regularized parameter $0<\varepsilon\ll1$. Two finite difference methods: SIFD, EFD were proposed and analyzed theoretically for the RLogKGE, which showed that the orders of accuracy are all second in both space and time. Besides, The numerical results demonstrated that the solutions of the RLogKGE (\ref{RLogKGE}) are linearly convergent to the LogKGE (\ref{LogKGE}) at $O(\varepsilon)$. In addition, the error bounds of FDTD methods to the LogKGE (\ref{LogKGE}) were numerically investigated and depended on $\tau,h,\varepsilon$.
\section*{Acknowledgments}
This work is supported by the National Natural Science Foundation of China (Grant No. 11971481,11901\\577), the Natural Science Foundation of Hunan (Grant No.S2017JJQNJJ0764,  S2020JJQNJJ1615), the Basic Research Foundation of National Numerical Wind Tunnel Project (No. NNW2018-ZT4A08). Research Fund of NUDT (Grand No. ZK17-03-27,ZK19-37), and the fund from Hunan Provincial Key Laboratory of Mathematical Modeling and Analysis in Engineering (Grand No.2018MMAEZD004).
\section*{References}
\bibliographystyle{elsarticle-num}
\bibliography{rlogkg_sifd}

\end{document}